\documentclass{article}
\usepackage[utf8]{inputenc}
\usepackage{amsmath,amsthm,amssymb,amsfonts}
\usepackage{hyperref,color}
\usepackage{tikz}
\def\a{\mathbf{a}}
\def\b{\mathbf{b}}

\def\m{\mathbf{m}}
\def\n{\mathbf{n}}

\def\Be{\mathcal{B}}
\def\Z{\mathbb{Z}}
\def\C{\mathbb{C}}
\def\N{\mathbb{N}}

\def\aalpha{\boldsymbol{\alpha}}
\def\bbeta{\boldsymbol{\beta}}

\newtheorem{theorem}{Theorem}[section]

\newtheorem{lemma}[theorem]{Lemma}
\newtheorem{prop}[theorem]{Proposition}
\newtheorem{corollary}[theorem]{Corollary}
\newtheorem{remark}[theorem]{Remark}
\newtheorem{example}[theorem]{Example}

\DeclareMathOperator*{\res}{res}
\newcommand{\floor}[1]{\left\lfloor #1 \right\rfloor}
\title{General duality relations for hypergeometric and basic hypergeometric series
\footnotetext{\textit{\textnormal{2020} AMS Mathematics Subject Classification: 33C20, 33D15}
\\ \textit{Key words and phrases:  hypergeometric function, basic hypergeometric function, duality relation, hypergeometric identity, $q$-hypergeometric identity, multi-term hypergeometric identity}
}}

\author{Dmitrii Karp \thanks{Holon Institute of Technology, Holon, Israel; Institute of Mathematics and Informatics, Bulgarian Academy of Sciences, Sofia, Bulgaria. Email: dimkrp@gmail.com}, Yi Zhang \thanks{Corresponding author. Department of Applied Mathematics, School of Mathematics and Physics, Xi'an Jiaotong-Liverpool University, 
 Suzhou, 215123, China.  Email: Yi.Zhang03@xjtlu.edu.cn. The work of Y.\ Zhang was supported by the NSFC grant No.\ 12371520. }}
\date{\today}

\begin{document}
\maketitle

\begin{abstract}
Duality relations for hypergeometric functions have reappeared as an active research topic several times, with the first instances tracing back to Euler and Gauss and the latest burst of activity occurring between 2015 and 2023.  In this paper we present a common generalization of all relations of this type found in the existing literature both for  hypergeometric and for $q$-hypergeometric functions.  We cover both Gauss type and confluent generalized hypergeometric functions and their  $q$-analogues. Our results entail a number of corollaries including multi-term relations for hypergeometric and $q$-hypergeometric series at a fixed argument. 
\end{abstract}

\section{Introduction and preliminaries}

Duality relations in the context of this paper are identities of the form $\sum_{j=1}^{p}c_{j}f_{j}(z)g_j(z)=\alpha(z)$, where $f_{j}$, $g_{j}$ are each expressed in terms of a single hypergeometric function ${}_{p}F_{q}$ or a single $q$-hypergeometric function ${}_{p}\phi_{q}$ dependent on interrelated parameters; $c_j$ are constants which also depend on parameters but not on the variable  $z$, and $\alpha(z)$ is an explicit elementary function, like power times a Laurent polynomial or its $q$-analogue.  The term ''duality relations'' originally refers to the situation when $f_{j}$, $g_{j}$ are solutions of adjoint homogeneous differential  (or difference) equations, but we retain it for a more general situation.

The history of duality relations for hypergeometric series can be traced back to Euler and Gauss with important contributions made by Bailey and Darling in the 1930s,  see details in \cite[Introduction]{KK}. More recent developments  include works by Nesterenko \cite[Theorem~5]{Nesterenko1995} and  Gorelov \cite[Corollary~1]{Gorelov2010}, motivated by number theory, by Feng, Kuznetsov and Yang \cite{FKY2}, motivated by computation of fractional Laplacian 
and by Beukers and Jouhet \cite{BJ}, derived by purely algebraic methods based
on the theory of $D$-modules of differential and difference equations. 

Another line of research driven by computation of the coefficients of contiguous relations for ${}_2F_{1}$ and ${}_2\phi_{1}$ functions is due to Ebisu \cite{Ebisu} and Yamaguchi \cite{Yamaguchi}.  Their relations are more general than specializations to ${}_2F_1$ (or ${}_2\phi_{1}$) of the relations due to Feng, Kuznetsov, Yang and those due to Beukers and Jouhet. In the paper \cite{KK} Alexey Kuznetsov and the first author found an extension of the relation from \cite{FKY2} which included Ebisu's ${}_2F_1$ identity. In the subsequent paper  \cite{KKK} Kalmykov, Kuznetsov and the first author found a $q$-analogue that included as a particular case Yamaguchi's formula. The identities from \cite{KK,KKK} did not include, however, the Beukers-Jouhet duality relations from \cite{BJ} neither for the ordinary nor for the basic hypergeometric functions.  The goal of this paper is to present identities that include all previously mentioned formulas as particular cases.

To set the stage, let us  introduce some notations.  We will use the standard Pochhammer symbol $(a)_{n}=\Gamma(a+n)/\Gamma(a)$ and the abbreviation $(\a)_n$ to denote the product  $(\a)_n=(a_1)_n(a_2)_n\cdots(a_p)_n$, where $n$ is an integer and $\a=(a_1,\ldots,a_p)$ is a $p$-tuple of real or complex numbers.  If $\n=(n_1,\ldots,n_p)$ has the same size as $\a$, then we set
$(\a)_{\n}=(a_1)_{n_1}(a_2)_{n_2}\cdots(a_p)_{n_p}$.
The symbol $\a_{[k]}$ will denote the $p$-tuple $\a$ with $k$-th component omitted.    The standard notation ${}_{p}F_q\left(\a;\b;z\right)$ will be used for the generalized hypergeometric series,
\begin{equation*}\label{eq:pFqdefined}
{}_{p}F_q\left(\left.\!\!\begin{array}{c}\a \\ \b\end{array}\right|z\!\right)={}_{p}F_q\left(\a;\b;z\right)
=\sum\limits_{n=0}^{\infty}\frac{(a_1)_n(a_2)_n\cdots(a_{p})_n}{(b_1)_n(b_2)_n\cdots(b_q)_nn!}z^n=\sum\limits_{n=0}^{\infty}\frac{(\a)_n}{(\b)_n n!}z^n.
\end{equation*}
The expression $\a+\alpha$  will be understood as the component-wise additon, where $\alpha$ is a scalar.

Let us now recall the Beukers-Jouhet formula from \cite{BJ}. Define $\vartheta=z\dfrac{d}{dz}$ to be Euler's operator. Set $b_r=1$ and $k,l=0,1,\ldots,r-1$.  Then the identity  \cite[Theorem~1.1]{BJ} takes the form
\begin{equation}\label{eq:BeukersJouhet}
\sum_{i=1}^r \frac{1}{(\b_{[i]}-b_i)_{1}}\vartheta^{k}z^{1-b_i}{_{r}F_{r-1}}\!\left(\left.\!\!\begin{array}{c}1+\a-b_i\\1+\b_{[i]}-b_i\end{array}\right|z\!\right)\vartheta^{l}z^{b_i-1}{_{r}F_{r-1}}\!\left(\left.\!\!\begin{array}{c}b_i-\a\\1-\b_{[i]}+b_i\end{array}\right|z\!\right) := M_{k l} \in H(z),
\end{equation}
where $b_i$'s are assumed to be distinct modulo integers, and $H$ is the field generated over $\mathbb{Q}$  by $a_i$, $b_j$.   Moreover, $M_{kl}=0$  if $k+l \leqslant r-2$ and  $M_{kl}=(-1)^k /(1-z)$  if  $k+l=r-1$. The functions
$$
z^{1-b_i}{_{r}F_{r-1}}\!\left(\left.\!\!\begin{array}{c}1+\a-b_i\\1+\b_{[i]}-b_i\end{array}\right|z\!\right)~~\text{and}~~z^{b_i-1}{_{r}F_{r-1}}\!\left(\left.\!\!\begin{array}{c}b_i-\a\\1-\b_{[i]}+b_i\end{array}\right|z\!\right)
$$
appearing in the above identity and related by $a_i\to1-a_i$, $b_i\to2-b_i$ form the bases of solutions around $z=0$ of the standard and dual hypergeometric differential equations, see \cite[(3.1)]{BJ}.  This explains the term ``duality relations''. 

A natural generalization of \eqref{eq:BeukersJouhet} can be obtained by replacing the operator $\vartheta$ by the operator  $\vartheta(\alpha)=z^{1-\alpha}\dfrac{d}{dz}z^{\alpha}$ dependent on a real parameter $\alpha$, so that $\vartheta=\vartheta(0)$, and  further replacing $\vartheta^k$ by 
\begin{equation}\label{eq:varthetaalpha}
\vartheta(\aalpha)=\vartheta(\alpha_{k},\ldots,\alpha_1):=\vartheta(\alpha_{k})\cdots\vartheta(\alpha_{1}),
\end{equation}
and, similarly, for $\vartheta(\bbeta)=\vartheta(\beta_{l},\ldots,\beta_1)$. Then the left hand side of \eqref{eq:BeukersJouhet} is replaced by 
\begin{equation}\label{eq:BeukersJouhetGen}
\sum\limits_{i=1}^{r}\frac{1}{(\b_{[i]}-b_i)_{1}}
\vartheta(\aalpha)z^{1-b_i}{_{r}F_{r-1}}\!\left(\left.\!\!\begin{array}{c}1+\a-b_i\\1+\b_{[i]}-b_i\end{array}\right|z\!\right)
\vartheta(\bbeta)z^{b_i-1}
{_{r}F_{r-1}}\!\left(\left.\!\!\begin{array}{c}b_i-\a,\\1-\b_{[i]}+b_i\end{array}\right|z\!\right).
\end{equation}
In Corollary~\ref{cr:beukersJouhetGen} below we will present 
the closed-form of the above formula in terms of rational functions in $z$.  We will also remove the restriction $k,l\le r-1$ imposed in \cite{BJ}.  
This result is a corollary of a more general new identity which also extends the relation from \cite[Theorem~1]{KK}, namely
\begin{equation}\label{eq:KarpKuznetsov}
\sum\limits_{i=1}^{r}\frac{(1-\b+a_i)_{\m-n_i}z^{-n_i}}{(a_i-\a_{[i]})_{\n_{[i]}-n_i+1}}
{}_{r}F_{r-1}\bigg(\begin{matrix}\b-a_i\\1+\a_{[i]}-a_i\end{matrix}\Big\vert z\bigg)
{}_{r}F_{r-1}\bigg(\begin{matrix}1-\b+a_i+\m-n_i\\1-\a_{[i]}+a_i+\n_{[i]}-n_i\end{matrix}\Big\vert z\bigg)
=R(z),
\end{equation}
where $\m, \n\in\Z^r$ and $R(z)$ is a rational function of $z$ given explicitly.  The common extension of \eqref{eq:BeukersJouhetGen} and \eqref{eq:KarpKuznetsov} will be given in Theorem~\ref{THM:generalizedBeukersJouhet} in the subsequent section.

To discuss the $q$-case, let us  recall some standard notations. Let $q$ be a fixed complex parameter (the `base') with $0<|q|<1$.  
The $q$-shifted factorial is defined for any complex parameter $a$ and a nonnegative integer $n$ by
\begin{equation} \label{EQ:qPochhammerdef}
(a; q)_0 = 1, \quad (a; q)_n = \prod_{k = 0}^{n - 1} (1 - a q^k), \quad (a; q)_{-n} = \prod_{k = 1}^n \frac{1}{1 - a/q^k}=\frac{1}{(aq^{-n};q)_{n}}=\frac{(-1)^{n}q^{n(n+1)/2}}{a^n(a^{-1}q;q)_{n}}, \quad n \in \N.
\end{equation}
The condition $|q|<1$ also guarantees that for any  complex $a$ the limit
\[
(a; q)_\infty = \lim_{n \rightarrow \infty} (a; q)_{n}
\]
exists as a finite number. Let $\a = (a_1, \ldots, a_r) \in \C^r$.  We will use the abbreviations 
$$
q^{\a} = (q^{a_1}, \cdots, q^{a_r}),\quad (\a;q)_{n} = (a_1; q)_{n} (a_2; q)_{n} \cdots (a_r; q)_{n},\quad 
(\a;q)_{\n} = (a_1; q)_{n_1} (a_2; q)_{n_2} \cdots (a_r; q)_{n_r},
$$
where  $\n = (n_1, \ldots, n_r) \in \Z^r$.
Furthermore, for an  integer $r\ge2$ and complex parameters $a_1, \ldots, a_r, b_1, \ldots, b_{r-1}$,  the  basic hypergeometric series is defined by
$$
{}_r\phi_{r-1}\left(\!\!\begin{array}{l}
\a \\
\b
\end{array} ; q, z\right):=\sum_{n=0}^{\infty} \frac{\left(a_1; q\right)_n \cdots\left(a_r; q\right)_nz^{n}}{\left(b_1; q\right)_n \cdots\left(b_{r-1}; q\right)_n\left(q; q\right)_n}=\sum_{n=0}^{\infty} \frac{(\a;q)_{n}z^{n}}{(\b;q)_n(q;q)_{n}}.
$$
Define $b_r=1$, $A=\sum_{i=1}^{r}a_i$ and $B=\sum_{i=1}^{r}b_i$ and suppose $k,l=0,1,\ldots,r-1$.  Slightly changing the notation \cite[Theorem~1.2]{BJ} can be stated as follows:
\begin{equation}\label{eq:q-BeukersJouhet}
\sum_{i=1}^{r} \frac{q^{1-b_i}}{(q^{\b_{[i]}-b_i})_{1}}
{}_r\phi_{r-1}\left(\!\!\begin{array}{c}
q^{1+\a-b_i} \\
q^{1+\b_{[i]}-b_i}
\end{array} ; q, q^{k}z\right){}_r\phi_{r-1}\left(\!\!\begin{array}{c}
q^{b_i-\a}\\
q^{1+b_i-\b_{[i]}}
\end{array}; q; q^{A-B+r-1+l}z\right)=M_{kl}(q) \in H_q(z),    
\end{equation}
where $H_q$ is the field generated by $q$, $q^{a_i}$, $q^{b_j}$ over $\mathbb{Q}$.  
Moreover, we have $M_{kl}(q)=0$ if $l \leqslant k \leqslant r-1$, except when $(k,l)=(r-1,0)$. In the latter case,
$$
M_{r-1,0}(q)=\frac{(-1)^{r+1}q^r}{q^{B}-q^{A+r-1}z}.
$$
In addition, we have
$$
M_{k,k+1}(q)=\frac{1}{1-q^kz}, \quad k=0,1,\ldots,r-2.
$$
In \cite[Theorem~1]{KKK} Kalmykov, Kuznetsov and the first author established a related identity of the form  
\begin{equation}\label{eq:q-KKK}
\sum_{i=1}^{r} q^{a_i(1-t)}\frac{(q^{1-\b+a_i};q)_{\m-n_i}z^{-n_i}}{(q^{a_i-\a_{[i]}};q)_{\n_{[i]}-n_i+1}}
{_r\phi_{r-1}}\left(\!\begin{array}{l}q^{\b-a_i}\\q^{1+\a_{[i]}-a_i}\end{array}\!\vline\, Wz\right)
{_r\phi_{r-1}}\left(\!\begin{array}{l}q^{1-\b+a_i+\m-n_i}\\q^{1-\a_{[i]}+a_i+\n_{[i]}-n_i}\end{array}\!\vline\,z\right)
=R_q(z),
\end{equation}
where $\a\in\mathbb{C}^r$ satisfies $a_i-a_j\notin\Z$ for $1\le i < j\le r$ and the vector $\b\in\mathbb{C}^{r}$ is arbitrary; $W=q^{t+r-1}\prod_{i=1}^r q^{a_i-b_i}$.
The function $R_q(z)$ is given explicitly in \cite[(8)]{KKK}.

In  Section~3 below we will present a duality relation that contains both \eqref{eq:q-BeukersJouhet} and \eqref{eq:q-KKK} as particular cases. It will be given in Theorem~\ref{THM: qBeukersJouhet2}.  
For both differential and for $q$-cases we will also provide confluent versions of our identities ({\it i.e.}, relations for the functions ${}_rF_{s}$  and  ${}_r\phi_{s}$ with $r\le{s}$) and a number of explicit examples of the right hand sides for small values of the components of $\mathbf{n}$, $\mathbf{m}$.  Furthermore, our results entail several  corollaries providing explicit expressions for certain finite sums of terminating and non-terminating hypergeometric and basic hypergeometric series at a fixed argument.   Finally, let us remark that all formulas presented in this paper have been verified numerically for particular values of parameters and should not contain typos or inaccuracies.

\section{The differential case}

Using the straightforward differentiation formula 
\begin{equation}\label{eq:xbetaDxalpha}
z^{\beta+1}\frac{d}{dz}\Bigg[z^{\alpha}{_{p}F_{q}}\left(\left.\!\!\begin{array}{c}\a\\\b\end{array}\right|z\!\right)\Bigg]=\alpha z^{\alpha+\beta}{_{p+1}F_{q+1}}\left(\left.\!\!\begin{array}{c}\a,\alpha+1\\\b,\alpha\end{array}\right|z\!\right)
\end{equation}
we can rewrite generalized Beukers-Jouhet  expression \eqref{eq:BeukersJouhetGen} as follows
\begin{multline*}\label{eq:SGeneral}
\sum\limits_{i=1}^{r}\frac{(1-b_i+\boldsymbol{\alpha})_{1}(b_i-1+\boldsymbol{\beta})_{1}}{(\b_{[i]}-b_i)_{1}}{_{r+k}F_{r+k-1}}\!\left(\left.\!\!\begin{array}{c}1+\a-b_i,2-b_i+\alpha_1,\ldots,2-b_i+\alpha_k\\1+\b_{[i]}-b_i,1-b_i+\alpha_1,\ldots,1-b_i+\alpha_k\end{array}\right|z\!\right)
\\
\times 
{_{r+l}F_{r+l-1}}\!\left(\left.\!\!\begin{array}{c}b_i-\a,b_i+\beta_1,\ldots,b_i+\beta_{l}\\1-\b_{[i]}+b_i,b_i+\beta_1-1,\ldots,b_i+\beta_{l}-1\end{array}\right|z\!\right).
\end{multline*}
Note that while the written formula \eqref{eq:xbetaDxalpha}
requires that $\alpha\ne0,-1,-2,\ldots$
it is in fact true for any $\alpha$ if we agree that  $\alpha(\alpha+1)_k/(\alpha)_k=(\alpha+k)$ also for $\alpha=-m$, $m=0,1,\ldots$.  The above form of the generalized Beukers-Jouhet  expression  provides a clue for simultaneous generalization of  \eqref{eq:BeukersJouhetGen} and \eqref{eq:KarpKuznetsov}. The result is presented in the following theorem. 
\begin{theorem}\label{THM:generalizedBeukersJouhet}
Suppose that $r\geq2$, $u,v\ge0$ are  integers, $\a,\b\in \C^r$, $\boldsymbol{\alpha}\in\C^{u}$, $\boldsymbol{\beta}\in\C^v$, and $\m,\n \in \Z^r$. Set 
\begin{equation} \label{EQ:mn}
M = \sum_{i=1}^r m_i, \quad N = \sum_{i = 1}^r n_i, \quad m_{\min} = \min_{1 \leq i \leq r} (m_i), \quad n_{\max} = \max_{1 \leq i \leq r} (n_i),
\end{equation}
\begin{equation} \label{EQ:indexp}
 p = \max \{-1, M-N+u+v-r+1\}.
\end{equation}
Assuming  that the components of $\b$ are distinct modulo integers, the following identity holds: 
\begin{multline} \label{EQ:generalizedBeukersJouhet} 
\sum_{i = 1}^r \frac{(1 + \mathbf{a} - b_i)_{\mathbf{m} - n_i} (1 - b_i + \boldsymbol{\alpha}-n_i)_1 (b_i - 1 + \boldsymbol{\beta})_1}{(\mathbf{b}_{[i]} - b_i)_{\mathbf{n}_{[i]} - n_i + 1}z^{n_i}}
\times {_{r + v}F_{r + v - 1}}\!\left(\left.\!\!\begin{array}{c} b_i - \mathbf{a}, b_i + \boldsymbol{\beta} \\ 1 - \mathbf{b}_{[i]}  +  b_i, b_i - 1 + \boldsymbol{\beta} \end{array} \right| z\!\right)
\\
\times {_{r+u}F_{r +u-1}}\!\left(\left.\!\!\!\begin{array}{c} 1 + \mathbf{a} - b_i + \mathbf{m} - n_i, 2 - b_i + \boldsymbol{\alpha}-n_i \\ 1 + \mathbf{b}_{[i]} - b_i + \mathbf{n}_{[i]} - n_i, 1 - b_i + \boldsymbol{\alpha}-n_i\end{array} \right| z\!\right) 
= (1 - z)^{- p - 1} \sum_{j = - n_{\max}}^{p - m_{\min}} \lambda_j z^j, 
\end{multline}
where the coefficient $\lambda_j$ is given by 
\begin{multline}\label{eq:beta-final}
\lambda_{j}=
\sum\limits_{\ell=\max(-n_{\max},j-p-1)}^{j}\binom{p+1}{j-\ell}(-1)^{j-\ell+u}\sum\limits_{i=1}^{r}
(b_i+\boldsymbol{\beta}-1)_{1}(b_i-\boldsymbol{\alpha}-\ell-1)_{1}
\\
\times\frac{(1+\a-b_i)_{\m+\ell}}{(\b_{[i]}-b_i)_{\n_{[i]}+\ell+1}(\ell+n_i)!}
F\!\left(\!\left.\!\begin{array}{l}-\ell-n_i,b_i-\a,1-\b_{[i]}+b_{i}-\n_{[i]}-\ell-1,b_i+\boldsymbol{\beta},b_i-\boldsymbol{\alpha}-\ell\\b_{i}-\a-\m-\ell,1-\b_{[i]}+b_{i},b_i+\boldsymbol{\beta}-1,b_i-\boldsymbol{\alpha}-\ell-1\end{array}\!\right.\right),
\end{multline}
with each term with $\ell+n_i<0$ vanishing by convention, $j = - n_{\max}, \ldots, p - m_{\min}$.
\end{theorem}

In order to prove the above theorem, we will need the following
\begin{lemma} \label{LEM:generalizedBeukersJouhet}
For each $k \in \Z$ define the rational function 
\begin{equation} \label{EQ:residualfunction}
f_k(z): = \frac{(z + \mathbf{a} - k + 1)_{\mathbf{m} + k} (z + \boldsymbol{\alpha} + 1)_1 (- z + \boldsymbol{\beta} + k - 1)_1}{(z + \mathbf{b} - k)_{\mathbf{n} + k + 1}}.
\end{equation}  
Then for $z\to\infty$ the function $f_k(z)$ has the asymptotic expansion
\begin{equation} \label{EQ:residualfunction2}
f_k(z) \sim \sum_{j=-\infty}^{M + u + v - N - r} C_j(k) z^j, 
\end{equation}
where $C_{-1}(k)$ is a polynomial in $k$ of degree $p$, which is defined by~\eqref{EQ:indexp} under the convention that a polynomial of negative degree vanishes.
\end{lemma}
\begin{proof}
It is straightforward to see from~\eqref{EQ:residualfunction} that $C_{-1}(k) = 0$ if $p < 0$ and $C_{-1}(k) = (-1)^v$ if $p = 0$. Assume that $p > 0$. In view of $(z)_k = \Gamma(z + k)/\Gamma(z)$, 
we have
\begin{multline} \label{EQ:logGamma}
\log[f_k(z)] = \sum_{i = 1}^r \{ \log \Gamma(z + a_i + m_i + 1) - \log \Gamma(z + a_i - k + 1) - \log \Gamma(z + b_i + n_i + 1) + \log \Gamma(z + b_i - k) \} \\
+ \sum_{\ell = 1}^u \log(z + \alpha_\ell + 1) + \sum_{\ell = 1}^v \log(- z + \beta_\ell + k - 1)
\end{multline}
Next,  recall  Hermite's asymptotic expansion for $\log\Gamma(z+a)$~\cite[(1.8)]{Nemes}
$$
\log\Gamma(z+a)\sim(z+a-1/2)\log{z}-z+\frac{1}{2}\log(2\pi)+\sum\limits_{j=2}^{\infty}\frac{(-1)^j\Be_{j}(a)}{j(j-1)z^{j-1}},
$$
as $|z|\to\infty$ in the domain $|\arg z|<\pi-\delta$,  $0<\delta<\pi$, where $\Be_j(x)$ is the $j$-th Bernoulli polynomial \cite[24.2.3]{NIST} generated by 
$$
\frac{t e^{x t}}{e^t-1}=\sum_{j=0}^{\infty} \Be_j(x) \frac{t^j}{j!}~~\text{and given by}~~\Be_j(x)=\sum_{l=0}^{j}\binom{j}{l}\Be_{j-l}x^l
$$
in terms of Bernoulli numbers $\Be_{j}$.  The equality $\Be_{0}=1$ implies that the leading coefficient of $\Be_j(x)$ is $1$. Substituting Hermite's expansion for each term in $\log[f_k(z)]$ into \eqref{EQ:logGamma} and using Taylor's expansion of the logarithm, we get 
\begin{equation} \label{EQ:logresidualfunction}
\log[f_k(z)] \sim (M+ u - N - r) \log(z) +  v \log(-z) + \sum\limits_{j=1}^{\infty}\frac{Q_j(k)}{z^j},
\end{equation}
where
\begin{multline*}
Q_{j}(k)=\frac{(-1)^{j+1}}{(j)_2}\sum\limits_{i=1}^{r}\Big[\Be_{j+1}(b_i - k) - \Be_{j+1}(a_i - k + 1) + \Be_{j+1}(a_i + m_i + 1) - \Be_{j+1}(b_i + n_i + 1)\Big]
\\ + (-1)^{j+1}\sum_{t = 1}^{u}(\alpha_t+1)^j-\sum_{t=1}^{v} (\beta_t +k-1)^j.
\end{multline*}
We see that $Q_{j}(k)$ is a polynomial in $k$ of degree $j$  (note that the leading coefficient of $\Be_{j+1}(x)$ is $1$, so that the degree $k+1$ terms   cancel out). 
Exponentiating  both sides of~\eqref{EQ:logresidualfunction}, we find (remember that $p>0$ by assumption)
\begin{equation} \label{EQ:residualfunction3}
f_k(z)\sim (-1)^v z^{p - 1}e^{\sum_{j\ge1 }Q_j(k)z^{-j}}.
\end{equation}
According to the well-known formula for an exponential of a power series  (see \cite[Lemma~1]{KPJMS2018} and references therein), we have 
\begin{equation} \label{EQ:expfps}
e^{\sum_{j \ge 1}Q_j(k)z^{-j}}=1+\sum_{s=1}^{\infty}q_s(k)z^{-s},
\end{equation}
where
$$
q_s(k)=\sum\limits_{\ell=1}^{s}\frac{1}{\ell!}\sum\limits_{\substack{{s_1+\cdots+s_\ell=s}\\{s_t\ge1}}}\prod\limits_{t=1}^{\ell}Q_{s_t}(k).
$$
The above expression shows that $q_s(k)$ is a polynomial in $k$ of degree at most $s$. Combining~\eqref{EQ:residualfunction2},~\eqref{EQ:residualfunction3} and~\eqref{EQ:expfps}, 
we conclude that $C_{-1}(k) = (-1)^v q_p(k)$, which is a polynomial in $k$ of degree at most $p$. 
\end{proof}

\noindent
\textit{Proof of Theorem~\ref{THM:generalizedBeukersJouhet}}. Using the Cauchy product expansion, we have 
\begin{align} \label{EQ:targetsum2}
S(z) & := \sum_{i = 1}^r \left\{ \frac{(1 + \mathbf{a} - b_i)_{\mathbf{m} - n_i} (1 - b_i + \boldsymbol{\alpha}-n_i)_1 (b_i - 1 + \boldsymbol{\beta})_1 z^{- n_i}}{(\mathbf{b}_{[i]} - b_i)_{\mathbf{n}_{[i]} - n_i + 1}} \right. \nonumber \\
 & \left. \times {_{r + u}F_{r + u - 1}}\left(\left.\!\!\begin{array}{c} 1 + \mathbf{a} - b_i + \mathbf{m} - n_i, 2 - b_i + \boldsymbol{\alpha}-n_i\nonumber  \\ 1 + \mathbf{b}_{[i]} - b_i + \mathbf{n}_{[i]} - n_i, 1 - b_i + \boldsymbol{\alpha}-n_i \end{array} \right| z\!\right) 
\times {_{r + v}F_{r + v - 1}}\left(\left.\!\!\begin{array}{c} b_i - \mathbf{a}, b_i + \boldsymbol{\beta} \\ 1 - \mathbf{b}_{[i]}  +  b_i, b_i - 1 + \boldsymbol{\beta} \end{array} \right| z\!\right) \right\} \nonumber \\
& = \sum_{i = 1}^r \left( \frac{(1 + \mathbf{a} - b_i)_{\mathbf{m} - n_i} z^{-n_i}}{(\mathbf{b}_{[i]} - b_i)_{\mathbf{n}_{[i]} - n_i + 1}} \right. \nonumber  \\
&\left. \sum_{k = 0}^\infty z^k \sum_{j = 0}^k \frac{(b_i - \mathbf{a})_j (b_i + \boldsymbol{\beta} + j - 1)_1 (1 + \mathbf{a} - b_i + \mathbf{m} - n_i )_{k - j} (1 - b_i + \boldsymbol{\alpha}-n_i + k - j)_1}{(1 - \mathbf{b}_{[i]} + b_i)_j (1 + \mathbf{b}_{[i]} - b_i + \mathbf{n}_{[i]} - n_i)_{k - j} j! (k - j)!} \right) \nonumber \\
& = \sum_{i = 1}^r \left( \sum_{k = 0}^\infty z^{k - n_i} \right. \nonumber  \\
& \left. \sum_{j = 0}^k \frac{(1 + \mathbf{a} - b_i)_{\mathbf{m} - n_i} (b_i - \mathbf{a})_j (b_i + \boldsymbol{\beta} + j - 1)_1 (1 + \mathbf{a} - b_i + \mathbf{m} - n_i )_{k - j} (1 - b_i + \boldsymbol{\alpha}-n_i + k - j)_1}{(\mathbf{b}_{[i]} - b_i)_{\mathbf{n}_{[i]} - n_i + 1} (1 - \mathbf{b}_{[i]} + b_i)_j (1 + \mathbf{b}_{[i]} - b_i + \mathbf{n}_{[i]} - n_i)_{k - j} j! (k - j)!} \right) \nonumber \\
& = \sum_{i = 1}^r \sum_{k = 0}^\infty z^{k - n_i} \sum_{j = 0}^k \gamma_{i, j}^k = \sum_{i = 1}^r \sum_{k_i = - n_i}^\infty z^{k_i} \sum_{j = 0}^{k_i + n_i} \gamma_{i, j}^{k_i + n_i},
\end{align}
where
\begin{equation} \label{EQ:gammaexpression1}
\gamma_{i, j}^k = \frac{(1 + \mathbf{a} - b_i)_{\mathbf{m} - n_i} (b_i - \mathbf{a})_j (b_i + \boldsymbol{\beta} + j - 1)_1 (1 + \mathbf{a} - b_i + \mathbf{m} - n_i )_{k - j} (1 - b_i + \boldsymbol{\alpha}-n_i + k - j)_1}{(\mathbf{b}_{[i]} - b_i)_{\mathbf{n}_{[i]} - n_i + 1} (1 - \mathbf{b}_{[i]} + b_i)_j (1 + \mathbf{b}_{[i]} - b_i + \mathbf{n}_{[i]} - n_i)_{k - j} j! (k - j)!}.
\end{equation}
Using the fact that $(z)_j = (-1)^j (1 - z - j)_j$, we have
\begin{equation} \label{EQ:simplifyfactorial1}
(1 + \mathbf{a} - b_i)_{\mathbf{m} - n_i} (b_i - \mathbf{a})_j (1 + \mathbf{a} - b_i + \mathbf{m} - n_i)_{k - j} = (-1)^{rj} (1 + \mathbf{a} - b_i - j)_{\mathbf{m} + k - n_i}.
\end{equation}
Similarly, we deduce that
\begin{equation} \label{EQ:simplifyfactorial2}
(\mathbf{b}_{[i]} - b_i)_{\mathbf{n}_{[i]} - n_i + 1} (1 - \mathbf{b}_{[i]} + b_i)_j (1 + \mathbf{b}_{[i]} - b_i + \mathbf{n}_{[i]} - n_i)_{k - j} = (-1)^{(r - 1) j} (\mathbf{b}_{[i]} - b_i - j)_{\mathbf{n}_{[i]} + k - n_i + 1}. 
\end{equation}
By~\eqref{EQ:gammaexpression1},~\eqref{EQ:simplifyfactorial1}, and~\eqref{EQ:simplifyfactorial2}, we have 
\[
\gamma_{i, j}^k = \frac{(-1)^j (1 + \mathbf{a}- b_i - j)_{\mathbf{m} + k - n_i} (b_i + \boldsymbol{\beta} + j - 1)_1 (1 - b_i + \boldsymbol{\alpha}-n_i+ k - j)_1}{(\mathbf{b}_{[i]} - b_i - j)_{\mathbf{n}_{[i]} + k - n_i + 1} j! (k - j)!},
\]
and 
\begin{equation}\label{eq:gammafinal}
\gamma_{i, j}^{k + n_i} = \frac{(-1)^j (1 + \mathbf{a}- b_i - j)_{\mathbf{m} + k} (b_i + \boldsymbol{\beta} + j - 1)_1 (1 - b_i + \boldsymbol{\alpha} + k  - j)_1}{(\mathbf{b}_{[i]} - b_i - j)_{\mathbf{n}_{[i]} + k + 1} j! (k + n_i - j)!}.    
\end{equation}
Furthermore, we set $\gamma_{i, j}^{k + n_i} = 0$ if $k + n_i < 0$. Using this convention, we may write~\eqref{EQ:targetsum2} as 
\[
S(z) = \sum_{k = - n_{\max}}^\infty z^k \sum_{i = 1}^r \sum_{j = 0}^{k + n_i} \gamma_{i, j}^{k + n_i},
\]
where $n_{\max}$ is specified in~\eqref{EQ:mn}. If $- n_{\max} \geq - m_{\min}$, then $k \geq - m_{\min}$ for each term in the above sum. Otherwise, if $- n_{\max} < - m_{\min}$, then we may write 
\[
S(z) = \sum_{k = - n_{\max}}^{-m_{\min} - 1} \mu_k z^k + S_1(z),
\] 
where
\[
S_1(z) = \sum_{k = - m_{\min}}^\infty \mu_k z^k, \quad \quad \quad \mu_k = \sum_{i = 1}^r \sum_{j = 0}^{k + n_i} \gamma_{i, j}^{k + n_i}.
\]
Thus, it suffices to show that $S_1(z)$ can be written as a rational function in the form~$z^{-m_{\min}} P_p(z)/(1 - z)^{p + 1}$, where $p$ is specified by~\eqref{EQ:indexp}, and $P_p(z)$ is a polynomial of degree at most $p$ with the convention that $P_{-1}(z) \equiv 0$. 
In order to achieve this, for each $k \in \Z$ we set 
\[
f_k(z) = \frac{(z + \mathbf{a} - k + 1)_{\mathbf{m} + k} (z + \boldsymbol{\alpha} +  1)_1 (- z + \boldsymbol{\beta} + k - 1)_1}{(z + \mathbf{b} - k)_{\mathbf{n} + k + 1}}, 
\]
which is the same function~\eqref{EQ:residualfunction} specified in Lemma~\ref{LEM:generalizedBeukersJouhet}. Note that $f_k(z)$ is well defined and rational for each integer $k$. Moreover, for $k \geq - m_{\min}$ all rising factorials in the product $(z + \mathbf{a} - k + 1)_{\mathbf{m} + k}$ are 
polynomials in $z$ and all poles of the rational function $f_k(z)$ are points such that $(z + b_i - k)_{k + n_i + 1} = 0$ with $k + n_i + 1 > 0$, {\it i.e.}, the points $z = - b_i + k - j$ for $j = 0, \ldots, k + n_i$. Since the components of $\mathbf{b}$ are distinct modulo integers, we see that 
all poles of $f_k(z)$ are simple. By a straightforward calculation, we have 
\begin{align*}
\res_{z = - b_i + k - j} f_k(z) & = \frac{(-1)^j (1 + \mathbf{a}- b_i - j)_{\mathbf{m} + k} (b_i + \boldsymbol{\beta} + j - 1)_1 (1 - b_i + \boldsymbol{\alpha} + k - j)_1}{(\mathbf{b}_{[i]} - b_i - j)_{\mathbf{n}_{[i]} + k + 1} j! (k + n_i - j)!} \\
& = \gamma_{i, j}^{k + n_i}.
\end{align*} 
Therefore, for all $k \geq - m_{\min}$ we have 
\[
\sum_{\text{over all poles of } f_k(z)} \res f_k(z) = \sum_{i = 1}^r \sum_{j = 0}^{k + n_i} \gamma_{i, j}^{k + n_i},
\]
where only the terms with $k + n_i \geq 0$ are non-vanishing. Using the fact that the sum of residues of a rational function at all finite points equals its residue at infinity, which is the coefficient at $z^{-1}$ in the asymptotic expansion
\begin{equation} \label{EQ:asymp2}
f_k(z) \sim \sum_{j = - \infty}^{M + u + v - N - r} C_j(k) z^j,~~~z\to\infty,
\end{equation}
which is the same as~\eqref{EQ:residualfunction2} in Lemma~\ref{LEM:generalizedBeukersJouhet}, we have 
\[
\sum_{i = 1}^r \sum_{j = 0}^{k + n_i} \gamma_{i, j}^{k + n_i} = C_{-1}(k).
\]
By Lemma~\ref{LEM:generalizedBeukersJouhet}, we see that $C_{-1}(k)$ is a polynomial in $k$ of degree at most $p$. Thus, we may write $C_{-1}(k) = \sum_{\ell = 0}^p w_\ell k^\ell$. 
Then we have 
\begin{align}\label{eq:S1C-1}
S_1(z) & = \sum_{k = - m_{\min}}^\infty z^k \sum_{i = 1}^r \sum_{j = 0}^{k + n_i} \gamma_{i, j}^{k + n_i}  = \sum_{k = - m_{\min}}^\infty z^k C_{-1}(k) \\
& = \sum_{k = - m_{\min}}^\infty z^k \sum_{\ell = 0}^p w_\ell k^\ell = z^{- m_{\min}} \sum_{\ell = 0}^p w_\ell \sum_{j = 0}^\infty (j - m_{\min})^\ell z^j. \nonumber
\end{align} 
It is straightforward to find that 
\[
\sum_{j = 0}^\infty (j - m_{\min})^\ell z^j = \frac{\bar{P}_\ell(z)}{(1  - z)^{\ell+1}},
\]
where $\bar{P}_\ell(z)$ is a polynomial in $z$ of degree at most $\ell$. Therefore, we have 
\begin{equation*}
S_1(z)  =  z^{- m_{\min}} \sum_{\ell = 0}^p w_\ell \frac{\bar{P}_\ell(z)}{(1  - z)^{\ell+1}}=: \frac{z^{- m_{\min}} P_p(z)}{(1 - z)^{p + 1}},
\end{equation*} 
where $P_p(z)$ is a polynomial in $z$ of degree at most $p$. Hence, we conclude that 
\begin{align} \label{EQ:lastformula}
S(z) & = \sum_{k = - n_{\max}}^{-m_{\min} - 1} \mu_k z^k + S_1(z) \\
& =  \sum_{k = - n_{\max}}^{-m_{\min} - 1} \mu_k z^k +   \frac{z^{- m_{\min}} P_p(z)}{(1 - z)^{p + 1}} \nonumber \\
& = (1 - z)^{- p - 1} \sum_{j = - n_{\max}}^{p - m_{\min}} \lambda_j z^j,  \nonumber
\end{align}
where $\lambda_j \in \C$ are certain coefficients defined for  $j = - n_{\max}, \ldots, p - m_{\min}$. Explicit formula \eqref{eq:beta-final} for $\lambda_j$ is established in Proposition~\ref{prop:beta-final} below. 
\hfill
\qedsymbol

\begin{remark}
Define two polynomial families \emph{(}$i=1,\ldots,r$\emph{):}
$$
G^{(i)}_v(y)=\prod\limits_{\ell=1}^{v}(b_i+\beta_{\ell}-1+y),~~~~H^{(i)}_u(y)=\prod\limits_{\ell=1}^{u}(1-b_i-n_i+\alpha_{\ell}+y).
$$
Then identity \eqref{EQ:generalizedBeukersJouhet} can be viewed as a polynomial perturbation  of \cite[Theorem~1]{KK} in the sense of  \cite{KarpAAM2026}. 
Namely, using the notation 
$$
F\!\left(\!\!\begin{array}{c}\mathbf{a}\\\mathbf{b}\end{array}\bigg\vert\,P_m\,\bigg\vert\, x\right):=\sum_{k=0}^{\infty}\frac{(\a)_{k}}{(\b)_kk!}P_m(k)x^k
$$
for perturbation of $F(\mathbf{a};\mathbf{b};x)$ by a polynomial $P_m$ of degree $m$, the left hand side of \eqref{EQ:generalizedBeukersJouhet} takes the form 
\begin{equation*}
\sum_{i = 1}^r \frac{(1 + \mathbf{a} - b_i)_{\mathbf{m} - n_i}}{(\mathbf{b}_{[i]} - b_i)_{\mathbf{n}_{[i]} - n_i + 1}z^{n_i}}
{F}\!\left(\!\!\begin{array}{c} b_i - \mathbf{a} \\ 1 - \mathbf{b}_{[i]}  +  b_i\end{array} \,\bigg\vert\,G^{(i)}_v\,\bigg\vert\, z\!\right)
{F}\!\left(\!\!\!\begin{array}{c} 1 + \mathbf{a} - b_i + \mathbf{m} - n_i \\ 1 + \mathbf{b}_{[i]} - b_i + \mathbf{n}_{[i]} - n_i\end{array} \,\bigg\vert\,H^{(i)}_u\,\bigg\vert\,z\!\right).
\end{equation*}
The original  identity~\cite[(3)]{KK} is recovered by setting the degrees to be zero: $u = v = 0$.
\end{remark}

An explicit expression \eqref{eq:beta-final} for the coefficients  $\lambda_k$ is established in the following proposition.
\begin{prop}\label{prop:beta-final}
The coefficient $\lambda_j$ on the right-hand side of \eqref{EQ:generalizedBeukersJouhet} is given by formula \eqref{eq:beta-final}.
\end{prop}
\begin{proof}
Writing $S(z)$ for the left hand side of \eqref{EQ:targetsum2}, we get by collecting terms
$$
S(z)=\sum\limits_{k=-n_{\max}}^{\infty}z^k
\underbrace{\sum\limits_{i=1}^{r}\sum\limits_{j=0}^{k+n_i}\gamma^{k+n_i}_{i,j}}_{=\delta_k}=\sum\limits_{k=-n_{\max}}^{\infty}\delta_{k}z^k.
$$
According to \eqref{eq:gammafinal}, we have the first equality below:
\begin{multline*}
\gamma_{i,j}^{k+n_i} = \frac{(-1)^j (1 + \mathbf{a}- b_i - j)_{\mathbf{m} + k} (b_i + \boldsymbol{\beta} + j - 1)_1 (1 - b_i + \boldsymbol{\alpha} + k  - j)_1}{(\mathbf{b}_{[i]} - b_i - j)_{\mathbf{n}_{[i]} + k + 1} j! (k + n_i - j)!}
\\
=\frac{(1+\a-b_i)_{\m+k}(b_i-\a)_{j}(1-\b_{[i]}+b_i-\n_{[i]}-k-1)_{j}(-k-n_i)_{j}}{(b_i-\a-\m-k)_{j}(\b_{[i]}-b_i)_{\n_{[i]}+k+1}(1-\b_{[i]}+b_i)_{j}(k+n_i)!j!}
\\
\times(-1)^{u}\frac{(b_i+\boldsymbol{\beta}-1)_{1}(b_i+\boldsymbol{\beta})_{j}(b_i-\boldsymbol{\alpha}-k-1)_{1}(b_i-\boldsymbol{\alpha}-k)_{j}}{(b_i+\boldsymbol{\beta}-1)_{j}(b_i-\boldsymbol{\alpha}-k-1)_{j}}.
\end{multline*}
The second equality above is obtained by an application of the following easily verifiable identities
$$
(z-j)_{n}=\frac{(z)_n(1-z)_{j}}{(1-z-n)_{j}}~~\text{and}~~(m-j)!=(-1)^j\frac{m!}{(-m)_{j}}.
$$
Hence,
\begin{multline}\label{eq:delta-defined}
\delta_k=
(-1)^{u}\sum\limits_{i=1}^{r}\frac{(b_i+\boldsymbol{\beta}-1)_{1}(b_i-\boldsymbol{\alpha}-k-1)_{1}(1+\a-b_i)_{\m+k}}{(\b_{[i]}-b_i)_{\n_{[i]}+k+1}(k+n_i)!}
\\
\times F\!\left(\!\left.\!\begin{array}{l}-k-n_i,b_i-\a,1-\b_{[i]}+b_{i}-\n_{[i]}-k-1,b_i+\boldsymbol{\beta},b_i-\boldsymbol{\alpha}-k\\b_{i}-\a-\m-k,1-\b_{[i]}+b_{i},b_i+\boldsymbol{\beta}-1,b_i-\boldsymbol{\alpha}-k-1\end{array}\!\right.\right).
\end{multline}
Multiplying both sides of \eqref{EQ:generalizedBeukersJouhet} by $(1-z)^{p+1}$ and expanding by the binomial theorem,  we obtain
$$
\sum_{j=0}^{p+1}\binom{p+1}{j}(-z)^j\sum\limits_{k=-n_{\max}}^{\infty}\delta_{k}z^k=\sum\limits_{k=-n_{\max}}^{p-m_{\min}}\lambda_{k}z^{k}.
$$
Multiplying both sides by $z^{n_{\max}}$, changing $k+n_{\max}\to{k}$ and writing $\hat{\lambda}_k=\lambda_{k-n_{\max}}$,
$\hat{\delta}_k=\delta_{k-n_{\max}}$, we get
$$
\sum_{j=0}^{p+1}\binom{p+1}{j}(-1)^jz^{j}\sum\limits_{k=0}^{\infty}\hat{\delta}_{k}z^k
=\sum\limits_{s=0}^{\infty}z^{s}\sum\limits_{j+k=s}\binom{p+1}{j}(-1)^{j}\hat{\delta}_{k}
=\sum\limits_{k=0}^{p-m_{\min}+n_{\max}}\hat{\lambda}_{k}z^{k}.
$$
In view of $\binom{p+1}{j}=0$ for $j>p+1$, this implies that
$$
\lambda_{s-n_{\max}}=\hat{\lambda}_s=\sum\limits_{j=0}^{\min(s,p+1)}\binom{p+1}{j}(-1)^{j}\hat{\delta}_{s-j}
=\sum\limits_{j=0}^{\min(s,p+1)}\binom{p+1}{j}(-1)^{j}\delta_{s-j-n_{\max}}
$$
for $s=0,\ldots,p-m_{\min}+n_{\max}$. Returning to $k=s-n_{\max}$, we obtain by changing the index of summation according to the rule $j\to{k-j}$:
$$
\lambda_{k}=\sum\limits_{j=0}^{\min(k+n_{\max},p+1)}\binom{p+1}{j}(-1)^{j}\delta_{k-j}=
\sum\limits_{j=\max(-n_{\max},k-p-1)}^{k}\binom{p+1}{k-j}(-1)^{k-j}\delta_{j}
$$
for $k=-n_{\max},\ldots,p-m_{\min}$. Substituting formula \eqref{eq:delta-defined} for $\delta_{j}$ and renaming indices, we finally arrive at  \eqref{eq:beta-final}.
\end{proof}

\noindent Below we present two explicit examples of the right-hand side of \eqref{EQ:generalizedBeukersJouhet} computed using \eqref{eq:beta-final} with $r = 3$ and $u=v=1$.
\begin{example} \label{EX: diffex1}
Set $\mathbf{m} = (1,1,2)$, $\mathbf{n} = (1,2,2)$,  $\boldsymbol{\alpha} = (1/2)$, and $\boldsymbol{\beta} = (1/3)$. Then the right-hand side of~\eqref{EQ:generalizedBeukersJouhet} takes the form 
\[
\frac{\left(-b_2-1/2\right) \left(b_2-2/3\right)}{(a_1- b_2) (a_2- b_2)
   (b_3- b_2) z^2}+\frac{\left(- b_3-1/2\right) \left(b_3-2/3\right)}{(a_1- b_3) ( a_2 - b_3) (b_2- b_3) z^2}.
\] 
\end{example}

\begin{example} \label{EX: diffex2}
Set $\mathbf{m} = (2,2,2)$, $\mathbf{n} = (0,2,3)$, $\boldsymbol{\alpha} = (1/5)$, and $\boldsymbol{\beta} = (1/7)$. Then the right-hand side of~\eqref{EQ:generalizedBeukersJouhet} takes the form 
\begin{multline*}
\frac{1}{(1 - z)^2 z^3}\left[ z^2 \left(\frac{\left(-b_3-9/5\right) \left(b_3-6/7\right) (b_1-b_3-2) (b_1-b_3-1)}{(a_1-b_3) (a_2-b_3)
   (a_3-b_3)} \right. \right.  \\ +\frac{1}{35} (-35 a_1-35 a_2-35 a_3+70 b_2+105 b_3-2)  \\
   \left. -\frac{2}{35} (-35 b_1+35 b_2+70 b_3+68)\right)+z
   \left(\frac{1}{35} (-35 b_1+35 b_2+70 b_3+68) \right. \\
 \left. \left. -\frac{2 \left(-b_3-9/5\right) \left(b_3-6/7\right) (b_1-b_3-2)
   (b_1-b_3-1)}{(a_1-b_3) (a_2-b_3) (a_3-b_3)}\right)+\frac{\left(-b_3-9/5\right) \left(b_3-6/7\right)
   (b_1-b_3-2) (b_1-b_3-1)}{(a_1-b_3) (a_2-b_3) (a_3-b_3)} \right]. 
\end{multline*}
\end{example}

The Beukers-Jouhet identity~\cite[Theorem 1.1]{BJ} is recovered from~\eqref{EQ:generalizedBeukersJouhet} by taking $\mathbf{m}, \mathbf{n}, \boldsymbol{\alpha}$ and $\boldsymbol{\beta}$ to be the zero vectors of the appropriate sizes. 
A generalization of Beukers-Jouhet identity hinted at in \eqref{eq:BeukersJouhetGen} can now be written explicitly by setting $\m=\n=\mathbf{0}$ in Theorem~\ref{THM:generalizedBeukersJouhet}.

\begin{corollary}\label{cr:beukersJouhetGen}
Suppose $\vartheta(\boldsymbol{\alpha})$ is defined in \eqref{eq:varthetaalpha},  $u,v\ge0$ are  integers, $\a,\b\in \C^r$, the components of $\b$ are distinct modulo integers, and $\boldsymbol{\alpha}\in\C^{u}$, $\boldsymbol{\beta}\in\C^v$. Then
\begin{equation*}
\sum\limits_{i=1}^{r}\frac{1}{(\b_{[i]}-b_i)_{1}}
\vartheta(\aalpha)z^{1-b_i}{_{r}F_{r-1}}\!\left(\left.\!\!\begin{array}{c}1+\a-b_i\\1+\b_{[i]}-b_i\end{array}\right|z\!\right)
\vartheta(\bbeta)z^{b_i-1}
{_{r}F_{r-1}}\!\left(\left.\!\!\begin{array}{c}b_i-\a,\\1-\b_{[i]}+b_i\end{array}\right|z\!\right)=(1 - z)^{-p-1} \sum_{j =0}^{p} \lambda_j z^j,
\end{equation*}
where $p = \max\{-1, u+v-r+1\}$. Moreover, if $p\ge0$, we have 
\begin{multline*}
\lambda_{j}=
\sum\limits_{\ell=0}^{j}\binom{p+1}{j-\ell}(-1)^{j-\ell+u}\sum\limits_{i=1}^{r}
(b_i+\boldsymbol{\beta}-1)_{1}(b_i-\boldsymbol{\alpha}-\ell-1)_{1}
\\
\times\frac{(1+\a-b_i)_{\ell}}{(\b_{[i]}-b_i)_{\ell+1}\ell!}
F\!\left(\!\left.\!\begin{array}{l}-\ell,b_i-\a,1-\b_{[i]}+b_{i}-\ell-1,b_i+\boldsymbol{\beta},b_i-\boldsymbol{\alpha}-\ell\\b_{i}-\a-\ell,1-\b_{[i]}+b_{i},b_i+\boldsymbol{\beta}-1,b_i-\boldsymbol{\alpha}-\ell-1\end{array}\!\right.\right).
\end{multline*}
\end{corollary}

Define 
\begin{equation}\label{eq:munu-defined}
\begin{split}
&\nu=\overbrace{\sum\nolimits_{k=1}^{r}b_k}^{=B}-\overbrace{\sum\nolimits_{k=1}^{r}a_k}^{=A}+v-r+1,
\\
&\mu=B-A+N-M-u-1.
\end{split}
\end{equation}
The following result has not previously appeared even for particular cases of identity \eqref{eq:BeukersJouhetGen}. 
\begin{corollary}\label{cr:mutitermat1}
Suppose that $\mu,\nu>0$.  Then 
\begin{multline}\label{eq:mutitermat1} 
\sum_{i = 1}^r \frac{\Gamma(1 + \mathbf{a} - b_i+\mathbf{m} - n_i) (1 - b_i + \boldsymbol{\alpha}-n_i)_1 }{\Gamma(\mathbf{b}_{[i]} - b_i+\mathbf{n}_{[i]} - n_i + 1)}
\frac{\sin[\pi(b_i-\mathbf{a})]}{\sin[\pi(\mathbf{b}_{[i]} - b_i)]}
\\
\times {_{r+u}F_{r +u-1}}\!\left(\left.\!\!\!\begin{array}{c} 1 + \mathbf{a} - b_i + \mathbf{m} - n_i, 2 - b_i + \boldsymbol{\alpha}-n_i \\ 1 + \mathbf{b}_{[i]} - b_i + \mathbf{n}_{[i]} - n_i, 1 - b_i + \boldsymbol{\alpha}-n_i\end{array} \right| 1\!\right) 
=0.
\end{multline}
\end{corollary}

\begin{remark}
Formula \eqref{eq:mutitermat1} does not seem to reduce to known multi-term identities for non-terminating hypergeometric series evaluated at unity, see \cite[(5.1),(5.3)]{CKP2021}.    
\end{remark}

\begin{proof}
We will use the following
asymptotic relation
\begin{equation}\label{eq:asympz1}
\lim\limits_{z\to1}\, (1-z)^{\psi}{}_{r}F_{r-1}\bigg(\!\left.\begin{matrix}\mathbf{c}\\\mathbf{d}\end{matrix}\,\right\vert\, z\bigg)=\frac{\Gamma(\mathbf{d})\Gamma(\psi)}{\Gamma(\mathbf{c})},
\end{equation}
which is valid when  $\psi=\sum_{j=1}^{r}c_j-\sum_{j=1}^{r-1}d_j>0$.
Note that the parametric excess (sum of the bottom parameters  minus sum of the top parameters) of the first hypergeometric factor in \eqref{EQ:generalizedBeukersJouhet} is precisely $-\nu$, which is defined by~\eqref{eq:munu-defined}, while the second factor is $\mu$.  By assumption, we have $\mu,\nu>0$ and $\nu-p-1\ge0$, where $p$ is defined in \eqref{EQ:mn}.  Multiplying \eqref{EQ:generalizedBeukersJouhet} by $(1-z)^{\nu}$ and taking the limit $z\to1$ in view of \eqref{eq:asympz1} (with $\psi=\nu>0$), we arrive at \eqref{eq:mutitermat1}.
\end{proof}

\medskip

Next, we turn our attention to the confluent case of Theorem~\ref{THM:generalizedBeukersJouhet} for  the hypergeometric function ${_{p}F_{q}}$ with $0 \leq p \leq q$, $p \geq 0$.  
Denote by $\floor{x}$ the floor function, whose value is defined to be the greatest integer less than or equal to the real number $x$. 

\begin{theorem}\label{THM:generalizedBeukersJouhet2}
Suppose that  $0\le s \le r - 1$, $r\ge2$, $u,v\ge0$ are integers,  $\mathbf{a}\in \C^s$, $\mathbf{b} \in \C^r$ has components distinct modulo integers,  $\mathbf{m}\in \Z^s$,  $\mathbf{n}\in \Z^r$, and $\boldsymbol{\alpha}\in\C^{u}$, $\boldsymbol{\beta}\in\C^v$. Set
\begin{equation}\label{eq:mnconfluent}
M = \sum_{i=1}^s m_i, \quad N = \sum_{i = 1}^r n_i, \quad m_{\min} = \min_{1 \leq i \leq s} (m_i), \quad n_{\max} = \max_{1 \leq i \leq r}(n_i), 
\end{equation}
which is the same as \eqref{EQ:mn} up to the new size of $\mathbf{m}$, and let 
\begin{equation}\label{eq:p_confluent}
p' = \floor{(M + u + v- N - r + 1)/(r - s)}.
\end{equation}
Then the following identity holds: 
\begin{multline}\label{eq:generalizedBeukersJouhet_confluent}
\sum_{i = 1}^r \frac{(1 + \mathbf{a} - b_i)_{\mathbf{m} - n_i} (1 - b_i + \boldsymbol{\alpha}-n_i)_1 (b_i - 1 + \boldsymbol{\beta})_1}{(\mathbf{b}_{[i]} - b_i)_{\mathbf{n}_{[i]} - n_i + 1} z^{n_i}} 
\times {_{s + v}F_{r + v - 1}}\left(\left.\!\!\begin{array}{c} b_i - \mathbf{a}, b_i + \boldsymbol{\beta} \\ 1 - \mathbf{b}_{[i]}  +  b_i, b_i - 1 + \boldsymbol{\beta} \end{array} \right| (-1)^{r - s} z\!\right)
\\
\times{_{s + u}F_{r + u - 1}}\!\left(\left.\!\!\!\begin{array}{c} 1 + \mathbf{a} - b_i + \mathbf{m} - n_i, 2 - b_i + \boldsymbol{\alpha}-n_i \\ 1 + \mathbf{b}_{[i]} - b_i + \mathbf{n}_{[i]} - n_i, 1 - b_i + \boldsymbol{\alpha}-n_i \end{array} \right| z\!\right)
= \sum_{k = - n_{\max}}^{\max(- m_{\min} - 1, p')} \delta_k z^k,
\end{multline}
where $\delta_k \in \C$ is given by 
\begin{multline*}
\delta_k=
(-1)^{u}\sum\limits_{i=1}^{r}\frac{(b_i+\boldsymbol{\beta}-1)_{1}(b_i-\boldsymbol{\alpha}-k-1)_{1}(1+\a-b_i)_{\m+k}}{(\b_{[i]}-b_i)_{\n_{[i]}+k+1}(k+n_i)!}
\\
\times F\!\left(\!\left.\!\begin{array}{l}-k-n_i,b_i-\a,1-\b_{[i]}+b_{i}-\n_{[i]}-k-1,b_i+\boldsymbol{\beta},b_i-\boldsymbol{\alpha}-k\\b_{i}-\a-\m-k,1-\b_{[i]}+b_{i},b_i+\boldsymbol{\beta}-1,b_i-\boldsymbol{\alpha}-k-1\end{array}\!\right.\right)
\end{multline*}
 for  $k = - n_{\max}, \ldots, \max(- m_{\min} - 1, p')$.
\end{theorem}

\begin{proof}
Follow the proof of Theorem~\ref{THM:generalizedBeukersJouhet} up to formula \eqref{EQ:simplifyfactorial1} which now will have the same form except for $(-1)^{rj}$ being  replaced by $(-1)^{sj}$.  Nevertheless, the presence of $(-1)^{r-s}$ in the argument of the first hypergeometric factor in \eqref{eq:generalizedBeukersJouhet_confluent} leads to the same expression for $\gamma^{k+n_i}_{i,j}$ as presented in \eqref{eq:gammafinal}.  Next,  formula \eqref{EQ:asymp2} is replaced by  
\[
f_k(z) = \sum_{j = - \infty}^{M + u + v - N - r - (r - s)k} C_j(k) z^j \quad \text{ as } \quad z \rightarrow \infty.
\]
This implies that $C_{-1}(k) = 0$ if $M + u + v - N - r - (r - s)k < - 1$, {\it i.e.}, when  $(r - s) k> M + u + v - N - r + 1$. 
Let $p' = \floor{(M + u + v- N - r + 1)/(r - s)}$. Then by the above argument and \eqref{eq:S1C-1}, we have
\[
S_1(z) = \sum_{k = - m_{\min}}^\infty C_{-1}(k) z^k =
	\sum_{k = - m_{\min}}^{p'} C_{-1}(k) z^k.
\]
Note that if $p'<-m_{\min}$, then $S_1(z)=0$. Therefore, instead of the calculation in~\eqref{EQ:lastformula} from the proof of Theorem~\ref{THM:generalizedBeukersJouhet}, here we get
\[
S(z) = \sum_{k = - n_{\max}}^{- m_{\min} - 1} \mu_k z^k + S_1(z) := \sum_{k = - n_{\max}}^{\max(- m_{\min} - 1, p')} \delta_k z^k.
\]
The numbers $\delta_k$ were computed in \eqref{eq:delta-defined}.
\end{proof}

As $\delta_k=0$ for $k>\max(-m_{\min}-1,p')$, in view of \eqref{eq:delta-defined} we immediately obtain
\begin{corollary}
Under conditions of Theorem~\ref{THM:generalizedBeukersJouhet2},  for each $k>\max(-m_{\min}-1,p')$ we have
\begin{multline*}
\sum\limits_{i=1}^{r}\frac{(b_i+\boldsymbol{\beta}-1)_{1}(b_i-\boldsymbol{\alpha}-k-1)_{1}(1+\a-b_i)_{\m+k}}{(\b_{[i]}-b_i)_{\n_{[i]}+k+1}(k+n_i)!}
\\
\times F\!\left(\!\left.\!\begin{array}{l}-k-n_i,b_i-\a,1-\b_{[i]}+b_{i}-\n_{[i]}-k-1,b_i+\boldsymbol{\beta},b_i-\boldsymbol{\alpha}-k\\b_{i}-\a-\m-k,1-\b_{[i]}+b_{i},b_i+\boldsymbol{\beta}-1,b_i-\boldsymbol{\alpha}-k-1\end{array}\!\right.\right)=0.
\end{multline*}
\end{corollary}

\noindent Below we give two examples of the right-hand side of formula \eqref{eq:generalizedBeukersJouhet_confluent} for the choice  $s = 1$ and $r = 3$, $u=v=1$.

\begin{example} \label{EX: diffconfluentex1}
Set $\mathbf{m} = (1)$, $\mathbf{n} = (0,1,1)$, $\boldsymbol{\alpha} = (1/3)$, and $\boldsymbol{\beta} = (1/5)$. Then the right-hand side of~\eqref{eq:generalizedBeukersJouhet_confluent} takes the form $(15 b_2+15 b_3-17)/(15 z)$.
\end{example}

\begin{example} \label{EX: diffconfluentex2}
Set $\mathbf{m} = (2)$, $\mathbf{n} = (3,3,3)$,  $\boldsymbol{\alpha} = (1/7)$, and $\boldsymbol{\beta} = (1/11)$. Then the right-hand side of~\eqref{eq:generalizedBeukersJouhet_confluent} takes the form
\[
\frac{77 a_1^2 + 73 a_1 - 130}{77 (a_1-b_1) (a_1-b_2) (b_3-a_1) z^3}.
\]
\end{example}

\section{The basic case}

 The $q$-gamma function (see \cite[(1.10.1)]{GR04} and~\cite[21.16]{KC02}) is defined by 
\[
\Gamma_q(z) = (1 - q)^{1 - z} \frac{(q; q)_\infty}{(q^z; q)_\infty}
\]
for $|q| < 1$ and each complex $z$ such that $q^{z + n} \neq 1$ for all $n \in \N\cup\{0\}$. Comparing this definition with~\eqref{EQ:qPochhammerdef}, we get
\begin{equation} \label{EQ:qgammaformula}
\frac{\Gamma_q(z + k)}{\Gamma_q(z)} = \frac{(q^z; q)_k}{(1 - q)^k}
\end{equation}
for any integer $k$. We will also need the standard $q$-binomial coefficients defined by
\begin{equation} \label{EQ:Gaussexpansion}
\left[ \begin{matrix} n \\ j \end{matrix} \right]_{q} = \frac{(q; q)_n}{(q; q)_j (q; q)_{n - j}} = \frac{(q^{n - j + 1}; q)_j}{(q; q)_j},
\end{equation}
see, for example,~\cite[p. 24]{GR04} or~\cite[Chapter 7]{KC02}. For any real number $x$, we set $(x)_{+} = \max(0, x)$. We will retain the notation 
$m_{\min} = \min_{1 \leq i \leq r} (m_i)$, $n_{\min} = \min_{1 \leq i \leq r} (n_i)$,  $M = \sum_{i = 1}^r m_i$, $N = \sum_{i = 1}^r n_i$ introduced previously in \eqref{EQ:mn}, \eqref{EQ:indexp}.  In this section we change slightly the definition of $p$ from \eqref{EQ:mn} as follows:
\begin{equation}\label{eq:p-qcase}
p = \max(-1, M  - N+ v - r - t + 1),    
\end{equation}
 where $t$ is an additional integer parameter.

The following lemma was established in the course of the proof of~\cite[Lemma~1]{KKK}.
\begin{lemma} \label{LEM:qBeukersJouhet2}
Consider the function 
\begin{equation} \label{EQ:basicqfactorial2}
g_k(z; \gamma, \eta) = \frac{(\eta z; q)_k}{(\gamma z; q)_k},
\end{equation}
where $k \in \Z$ and $\gamma,\eta\in\C$. Then
\begin{itemize}
\item[(i)] For sufficiently large $z$ we have
\[
g_k(z; \gamma, \eta) = (\eta/\gamma)^k \sum_{\ell \geq 0} z^{-\ell} S_\ell(q^{- k}),
\]
where $S_\ell(w)$ is a polynomial of degree $\ell$ whose coefficients do not depend on $k$.
\item[(ii)] For sufficiently small $z$ we have 
\[
g_k(z; \gamma, \eta) = \sum_{\ell \geq 0} z^\ell T_\ell(q^k),
\]
where $T_\ell(w)$ is a polynomial of degree $\ell$ whose coefficients do not depend on $k$.
\end{itemize}
\end{lemma}

\begin{remark}
The proof of \cite[Lemma~1]{KKK} is written for $k\ge1$, but careful examination shows that  essentially the same proof works for negative $k$.  The trivial case $k=0$ is also  included since $S_{\ell}(1)=T_{\ell}(1)=0$ for $\ell\ge1$. 
\end{remark}

\begin{corollary} \label{COR:qBeukersJouhet2}
Suppose $\mathbf{a},\mathbf{b} \in \C^r$, $\mathbf{m}, \mathbf{n} \in \Z^r$, $t\in\Z$, $\boldsymbol{\alpha} \in \C^u$ and $\boldsymbol{\beta} \in \C^v$.  Consider the function
    \begin{equation} \label{EQ:targetfunction2}
    F_k(z) = - z^{- t} \frac{( z q^{1 - \mathbf{b}}; q)_{k + \mathbf{m}} (z q^{1 + \boldsymbol{\beta} + k}; q)_1 (z^{-1} q^{\boldsymbol{\alpha} - 1}; q)_1}{(z q^{- \mathbf{a}}; q)_{k + \mathbf{n}+1}},
    \end{equation}
    where $k \in \mathbb{Z}$. Then
    
    \begin{enumerate}
        \item[(i)] For sufficiently large $z$ we have
        
        \begin{equation} \label{EQ:desiredresult3}
        F_k(z) = (q^vB)^k z^{M + v - N - r - t} \sum_{\ell \geq 0} Q_\ell(q^{- k})z^{-\ell}, 
        \end{equation}
        where \( B = \prod_{i=1}^r q^{a_i - b_i + m_i - n_i} \) and for each $\ell \geq 0$ the function $Q_\ell(w)$ is a polynomial of degree $\ell$ whose coefficients do not depend on $k$.
        
        \item[(ii)] For sufficiently small $z$ we have
        \begin{equation} \label{EQ:desiredresult4}
        F_k(z) = z^{-u - t}\sum_{\ell \geq 0} P_\ell(q^k)z^{\ell},
        \end{equation}
      where  for each $\ell \geq 0$ the function $P_\ell(w)$ is a polynomial of degree \( \ell \) whose coefficients do not depend on \( k \).
    \end{enumerate}
\end{corollary}

\begin{proof}
Set
\[
\tilde{f}_k(z) = - \frac{( z q^{1 - \mathbf{b}}; q)_{k + \mathbf{m}} (z q^{1 + \boldsymbol{\beta} + k}; q)_1 (z^{-1} q^{\boldsymbol{\alpha} - 1}; q)_1}{(z q^{- \mathbf{a}}; q)_{k + \mathbf{n}+1}}.
\]
Then 
\begin{equation} \label{EQ:decomposion2}
F_k(z) = z^{-t} \tilde{f}_k(z).
\end{equation}
In terms of $g_k$ defined in~\eqref{EQ:basicqfactorial2} we can write
\begin{equation} \label{EQ:prodqfactorial2}
\tilde{f}_k(z) = R(z) \cdot \prod_{i = 1}^r g_k(z; q^{-a_i + n_i + 1}, q^{- b_i + m_i + 1})\prod_{j=1}^{v} g_k(z; q^{1+\beta_j}, q^{2+\beta_j}),
\end{equation}
where
\[
R(z) = - \frac{( z q^{1 - \mathbf{b}}; q)_{\mathbf{m}} (z q^{1 + \boldsymbol{\beta}}; q)_1 (z^{-1} q^{\boldsymbol{\alpha} - 1}; q)_1}{(z q^{- \mathbf{a}}; q)_{\mathbf{n} + 1}}.
\]
Note that the above factorization is true for all integer $k$, which can be verified using \eqref{EQ:qPochhammerdef}. From this expression we conclude that the function $R$ has the series expansion (for all sufficiently large $z$)
\begin{equation} \label{EQ:rseriesinfity2}
R(z) = z^{M + v - N - r} \sum_{\ell \geq 0} c_\ell z^{-\ell},
\end{equation}
where $c_\ell \in \C$ is certain coefficient independent of $k$, $c_0\ne0$. Using item (i) of Lemma~\ref{LEM:qBeukersJouhet2}, we find that for large $z$
\begin{equation} \label{EQ:prodseriesinfty2}
\prod_{i = 1}^r g_k(z; q^{-a_i + n_i + 1}, q^{- b_i + m_i + 1}) = B^k \sum_{\ell \geq 0} z^{-\ell} \tilde{Q}_\ell(q^{-k}),
\end{equation}
where \( B = \prod_{i=1}^r q^{a_i - b_i + m_i - n_i} \) and $\tilde{Q}_\ell$  are certain polynomials of degree $\ell$.
Similarly,
$$
\prod_{j=1}^{v} g_k(z; q^{1+\beta_j}, q^{2+\beta_j})=q^{vk}\sum_{\ell\ge0}z^{-\ell}\hat{Q}_{\ell}(q^{-k}).
$$
Combining~\eqref{EQ:decomposion2},~\eqref{EQ:prodqfactorial2},~\eqref{EQ:rseriesinfity2}, and~\eqref{EQ:prodseriesinfty2}, we 
arrive at the desired result~\eqref{EQ:desiredresult3}.

For small $z$, it is clear that $R$ has the series expansion
\begin{equation}   \label{EQ:rseriesinfity3}
R(z) = z^{-u} \sum_{\ell \geq 0} \tilde{c}_\ell z^\ell
\end{equation}
for some coefficients $\tilde{c}_\ell \in \C$. Using  item (ii) of Lemma~\ref{LEM:qBeukersJouhet2}, we have (for all $z$ small enough)
\begin{subequations}\label{EQ:prodseriesinfty3}
\begin{equation} 
\prod_{i = 1}^r g_k(z; q^{-a_i + n_i + 1}, q^{- b_i + m_i + 1}) = \sum_{\ell \geq 0} z^\ell \bar{P}_\ell(q^k),
\end{equation}
and 
\begin{equation} 
\prod_{j=1}^{v} g_k(z; q^{1+\beta_j}, q^{2+\beta_j})=\sum_{\ell\ge0}z^{\ell}\hat{P}_{\ell}(q^{k})
\end{equation}
\end{subequations}
where each function $\bar{P}_\ell$ and $\hat{P}_\ell$ is a polynomial of degree $\ell$. Combining~\eqref{EQ:decomposion2},~\eqref{EQ:prodqfactorial2},~\eqref{EQ:rseriesinfity3}, and~\eqref{EQ:prodseriesinfty3}, we 
arrive at the desired expansion~\eqref{EQ:desiredresult4}.
\end{proof}

Our main result is the following  

\begin{theorem} \label{THM: qBeukersJouhet2}
Suppose that $r \geq 2$, $u,v\ge0$ are integers, $q$ with $0 < |q| < 1$ is a complex number, 
the vector $\mathbf{a} \in \C^r$ satisfies $a_i - a_j \not\in \Z$ for $1 \leq i < j \leq r$, the vectors $\mathbf{b} \in \C^r$, and $\boldsymbol{\alpha} \in \C^u$, $\boldsymbol{\beta} \in \C^v$ 
are arbitrary. Assume further that $\mathbf{m}, \mathbf{n} \in \Z^r$ and $t \in \Z$.  
Let $W = q^{t + r -1} \prod_{i = 1}^r q^{a_i - b_i}$, $p$ be defined by \eqref{eq:p-qcase} and $m_{\min}$, $n_{\max}$ retain their meaning from \eqref{EQ:mn}. 
Then the following identity holds
\begin{align} \label{EQ:generalizedqBeukersJouhet}
& \sum_{i = 1}^r q^{a_i (1 - t)} \frac{(q^{1 - \mathbf{b} + a_i}; q)_{\mathbf{m} - n_i} (q^{- a_i - 1 + \boldsymbol{\alpha}}; q)_1 (q^{1 + a_i + \boldsymbol{\beta}-n_i}; q)_1 z^{-n_i}}{(q^{a_i - \mathbf{a}_{[i]}}; q)_{\mathbf{n}_{[i]} - n_i + 1}} {_{r + u}\phi_{r + u - 1}}\left(\left.\!\!\begin{array}{c} q^{\mathbf{b} - a_i}, q^{- a_i + \boldsymbol{\alpha}} \\ q^{1 + \mathbf{a}_{[i]} - a_i}, q^{- a_i - 1 + \boldsymbol{\alpha}} \end{array} \right| W z\!\right) \nonumber \\
& \times {_{r+v}\phi_{r+v-1}} \left(\left.\!\!\begin{array}{c} q^{1 - \mathbf{b} + a_i + \mathbf{m} - n_i}, q^{2 + a_i + \boldsymbol{\beta}-n_i} \\ q^{1 - \mathbf{a}_{[i]} + a_i + \mathbf{n}_{[i]} - n_i}, q^{1 + a_i + \boldsymbol{\beta}-n_i} \end{array} \right| z\!\right)  = \frac{1}{(Wz; q)_{p+1}(z; q)_{(u+t)_{+}}} \sum_{k=-n_{\max}}^{p+(u+t)_{+}-m_{\min}}\lambda_k z^k,
\end{align}
\noindent
 where the number $\lambda_k$ is given by
\begin{align} \label{EQ:explicitformulalambda}
\lambda_k = & \sum_{j = \max(- n_{\max}, k - p - (u + t)_{+} -1)}^k (-1)^{k - j} \nonumber \\
& \times \sum_{g + h = k - j} \left[ \begin{matrix} p + 1 \\  h \end{matrix} \right]_{q} \left[ \begin{matrix} (u + t)_{+} \\  g \end{matrix} \right]_{q} q^{(h (h - 1) + g (g -1))/2} W^h \nonumber \\
& \times \bigg\{ \sum_{i = 1}^r \frac{q^{a_i(1 - t)} (q^{1 + a_i + \boldsymbol{\beta} + j}; q)_1 (q^{\boldsymbol{\alpha}-a_i-1}; q)_1 (q^{1 - \mathbf{b} + a_i} ; q)_{\mathbf{m} + j}}{(q; q)_{j + n_i} (q^{a_i - \mathbf{a}_{[i]}}; q)_{\mathbf{n}_{[i]} + j + 1}} 
\nonumber \\ 
&  {_{2 r + u + v}\phi_{2 r + u + v - 1}}\left(\left.\!\!\begin{array}{c} q^{- j - n_i}, q^{\mathbf{b} - a_i}, q^{\mathbf{a}_{[i]} - a_i - \mathbf{n}_{[i]} - j}, q^{-a_i-\boldsymbol{\beta}-j}, q^{\boldsymbol{\alpha}-a_i} \\ q^{\mathbf{b} - a_i - \mathbf{m} - j}, q^{1 - a_i + \mathbf{a}_{[i]}}, q^{-1 - a_i - \boldsymbol{\beta} -  j }, q^{\boldsymbol{\alpha}- a_i-1} \end{array} \right| q^{N - M + r -v + t-1} \!\right)  \bigg\}.
\end{align}
\end{theorem}

\noindent
\textit{Proof of Theorem~\ref{THM: qBeukersJouhet2}}. By the Cauchy product,  we have 
\begin{align} \label{EQ:qtargetsum2}
S(z) & := \sum_{i = 1}^r q^{a_i (1 - t)} \frac{(q^{1 - \mathbf{b} + a_i}; q)_{\mathbf{m} - n_i} (q^{- a_i - 1 + \boldsymbol{\alpha}}; q)_1 (q^{1 + a_i + \boldsymbol{\beta}-n_i}; q)_1 z^{-n_i}}{(q^{a_i - \mathbf{a}_{[i]}}; q)_{\mathbf{n}_{[i]} - n_i + 1}} \nonumber \\
& \times {_{r + u}\phi_{r + u - 1}}\left(\left.\!\!\begin{array}{c} q^{\mathbf{b} - a_i}, q^{- a_i + \boldsymbol{\alpha}} \\ q^{1 + \mathbf{a}_{[i]} - a_i}, q^{- a_i - 1 + \boldsymbol{\alpha}} \end{array} \right| W z\!\right) \times {_{r + v}\phi_{r + v - 1}}\left(\left.\!\!\begin{array}{c} q^{1 - \mathbf{b} + a_i + \mathbf{m} - n_i}, q^{2 + a_i + \boldsymbol{\beta}-n_i} \\ q^{1 - \mathbf{a}_{[i]} + a_i + \mathbf{n}_{[i]} - n_i}, q^{1 + a_i + \boldsymbol{\beta}-n_i} \end{array} \right| z\!\right) \nonumber  \\
& = \sum_{i = 1}^r \sum_{k = 0}^\infty z^{k - n_i} \nonumber \\
&  \times \sum_{j = 0}^k \frac{ (q^{1 - \mathbf{b} + a_i}; q)_{\mathbf{m} - n_i} (q^{\mathbf{b} - a_i}; q)_j (q^{1 - \mathbf{b} + a_i + \mathbf{m} - n_i}; q)_{k - j} (q^{- a_i + \boldsymbol{\alpha} + j - 1}; q)_1 (q^{1 + a_i + \boldsymbol{\beta}-n_i+ k - j}; q)_1 W^j}{q^{a_i (t-1)}(q^{a_i - \mathbf{a}_{[i]}}; q)_{\mathbf{n}_{[i]} - n_i + 1}(q^{1 + \mathbf{a}_{[i]} - a_i}; q)_j (q^{1 - \mathbf{a}_{[i]} + a_i + \mathbf{n}_{[i]} - n_i}; q)_{k - j} (q; q)_j (q; q)_{k - j}} \nonumber \\
& = \sum_{i = 1}^r \sum_{k = 0}^\infty z^{k - n_i} \sum_{j = 0}^k  \gamma_{i, j}^k = \sum_{i = 1}^r \sum_{k_i = - n_i}^\infty z^{k_i} \sum_{j = 0}^{k_i + n_i} \gamma_{i, j}^{k_i + n_i}, 
\end{align}
where 
\begin{equation} \label{EQ:defgamma2}
\gamma_{i, j}^k =  \frac{q^{a_i (1 - t)} (q^{1 - \mathbf{b} + a_i}; q)_{\mathbf{m} - n_i} (q^{\mathbf{b} - a_i}; q)_j (q^{1 - \mathbf{b} + a_i + \mathbf{m} - n_i}; q)_{k - j} (q^{- a_i + \boldsymbol{\alpha} + j - 1}; q)_1 (q^{1 + a_i + \boldsymbol{\beta}-n_i + k - j}; q)_1 W^j}{(q^{a_i - \mathbf{a}_{[i]}}; q)_{ \mathbf{n}_{[i]} - n_i + 1}(q^{1 + \mathbf{a}_{[i]} - a_i}; q)_j (q^{1 - \mathbf{a}_{[i]} + a_i + \mathbf{n}_{[i]} - n_i}; q)_{k - j} (q; q)_j (q; q)_{k - j}}.
\end{equation}
Using \eqref{EQ:qgammaformula} and
\begin{equation} \label{EQ:qPochhammerformula}
(q^a; q)_n = (-1)^n q^{an + n(n-1)/2} (q^{1 - a - n};q)_n
\end{equation}
 we can simplify the $q$-Pochhammer symbols in the denominator on the right-hand side of~\eqref{EQ:defgamma2} as follows: 
\begin{align} \label{EQ:simplifyqPochhamer}
& \quad \ (q^{a_i - a_{\ell}}; q)_{n_{\ell} - n_i + 1}(q^{1 + a_{\ell} - a_i}; q)_j (q^{1 - a_{\ell} + a_i + n_{\ell} - n_i}; q)_{k - j} \nonumber \\
& = (q^{1 + a_\ell - a_i}; q)_j (1 - q)^{n_\ell - n_i + 1 + k - j} \frac{\Gamma_q(a_i - a_\ell + n_\ell - n_i + 1)}{\Gamma_q(a_i - a_\ell)} \frac{\Gamma_q(1 - a_\ell + a_i + n_\ell - n_i + k - j)}{\Gamma_q(1 - a_\ell + a_i + n_\ell - n_i)}  \nonumber\\
& =  (q^{1 + a_\ell - a_i}; q)_j (q^{a_i - a_\ell}; q)_{n_\ell - n_i + 1 + k - j}   \nonumber \\[6pt]
& = (-1)^j q^{(1 + a_\ell - a_i)j + j (j - 1)/2} (q^{- a_\ell + a_i - j}; q)_j (q^{a_i - a_\ell}; q)_{n_\ell - n_i + 1 + k - j}  \nonumber \\[6pt]
& =  (-1)^j q^{(1 + a_\ell - a_i)j + j (j - 1)/2} (q^{a_i - a_\ell - j}; q)_{n_\ell - n_i + 1 + k}.
\end{align}
Using~\eqref{EQ:simplifyqPochhamer} and doing a similar calculation for the numerator in~\eqref{EQ:defgamma2}, we get
\[
\gamma_{i, j}^k = \frac{(-1)^j q^{j (j - 1)/2 + a_i (1 - t) + t j} (q^{1 - \mathbf{b} + a_i - j}; q)_{\mathbf{m} + k - n_i} (q^{- a_i + \boldsymbol{\alpha} + j - 1}; q)_1 (q^{1 + a_i + \boldsymbol{\beta}-n_i + k - j}; q)_1}{(q^{a_i - \mathbf{a}_{\mathbf{[i]}} - j}; q)_{\mathbf{n}_{[i]} + k - n_i + 1} (q; q)_j (q; q)_{k - j}},
\]
which is equivalent to 
\begin{equation} \label{EQ:defgamma3}
\gamma_{i, j}^{k + n_i} = \frac{(-1)^j q^{j (j - 1)/2 + a_i(1 - t) + t j} (q^{1 - \mathbf{b} + a_i - j}; q)_{\mathbf{m} + k} (q^{- a_i + \boldsymbol{\alpha} + j - 1}; q)_1 (q^{1 + a_i + \boldsymbol{\beta} + k - j}; q)_1}{(q^{a_i - \mathbf{a}_{\mathbf{[i]}} - j}; q)_{\mathbf{n}_{[i]} + k + 1} (q; q)_j (q; q)_{k + n_i - j}}.
\end{equation}
Moreover, we set $\gamma_{i, j}^{k + n_i} = 0$ if $k + n_i < 0$. In view of this convention, we can write~\eqref{EQ:qtargetsum2} as
\begin{equation}\label{eq:Sintermsofgammas}
S(z) = \sum_{k = - n_{\max}}^\infty z^k \sum_{i = 1}^r \sum_{j = 0}^{k + n_i} \gamma_{i, j}^{k + n_i}= \sum_{k = - n_{\max}}^\infty \mu_kz^k,
\end{equation}
where the last equality is the definition of $\mu_k$. Whenever $- n_{\max} \geq - m_{\min}$, we have $k \geq - m_{\min}$ in the above sum. Otherwise, when $- n_{\max} < - m_{\min}$, 
we can write 
\begin{equation} \label{EQ:targetsum3}
S(z) = \sum_{k = - n_{\max}}^{- m_{\min} - 1} \mu_k z^k + S_1(z),
\end{equation}
where 
\[
S_1(z) = \sum_{k = - m_{\min}}^\infty \mu_k z^k, \quad \quad \mu_k = \sum_{i = 1}^r \sum_{j = 0}^{k + n_i} \gamma_{i, j}^{k + n_i}. 
\]
To show that $S_1(z)$ is a rational function of $z$, for each $k\in\Z$  we define the function
\begin{equation}\label{eq:fF-defined}
\tilde{f}_k(z) = - \frac{( z q^{1 - \mathbf{b}}; q)_{k + \mathbf{m}} (z q^{1 + \boldsymbol{\beta} + k}; q)_1 (z^{-1} q^{\boldsymbol{\alpha} - 1}; q)_1}{(z q^{- \mathbf{a}}; q)_{k + \mathbf{n}+1}} \quad \text{ and } \quad F_k(z) = z^{-t} \tilde{f}_k(z).    
\end{equation}
Clearly, $\tilde{f}_k(z)$ is  a rational function of  $z$ for each $k \in \Z$. Moreover, for $k\ge -m_{\min}$ all terms in the numerator except for $(z^{-1}q^{\boldsymbol{\alpha}-1};q)_{1}$ are polynomial in $z$ so that 
all nonzero poles of $\tilde{f}_k$ come from the zeros of the denominator as the only singularity of  $(z^{-1}q^{\boldsymbol{\alpha}-1};q)_{1}$ is a  multiple pole at $z=0$. When $k + n_i + 1 > 0$ for each $i = 1, \ldots, r$, the nonzero poles of $\tilde{f}_k(z)$ are 
points $z$ such that $(z q^{-\mathbf{a}}; q)_{k + \mathbf{n} + 1} = 0$, {\it i.e.}, 
\[
z = q^{a_i - j}, \quad \ i =1, \ldots, r, \quad \text{ and } \quad \ j = 0, \ldots, k + n_i. 
\]
Note that the terms with $k+n_i<0$ do not contribute to the sum of residues.  
Since $a_i - a_j \not\in \Z$ for $1 \leq i < j \leq r$, we see that all these poles are simple. Therefore, the finite poles of $F_k$ include all poles of $\tilde{f}_k$ and the pole at $z =0$ of order $t+u$ if $t+u>0$.   After a straightforward calculation, we get in view of \eqref{EQ:defgamma3} that 
\[
\res_{z = q^{a_i - j}} F_k(z) = \frac{(-1)^j q^{j (j - 1)/2 + a_i(1 - t) + t j} (q^{1 - \mathbf{b} + a_i - j}; q)_{\mathbf{m} + k} (q^{- a_i + \boldsymbol{\alpha} + j - 1}; q)_1 (q^{1 + a_i + \boldsymbol{\beta} + k - j}; q)_1}{(q^{a_i - \mathbf{a}_{\mathbf{[i]}} - j}; q)_{\mathbf{n}_{[i]} + k + 1} (q; q)_j (q; q)_{k + n_i - j}} = \gamma_{i, j}^{k + n_i}. 
\]
Hence, for each $k \geq - m_{\min}$ we have 
\[
\sum_{\text{over all finite nonzero poles of } F_k(z)} \res F_k(z) = \sum_{i = 1}^r \sum_{j = 0}^{k + n_i} \gamma_{i, j}^{k + n_i}, 
\]
where only the terms with $k + n_i \geq 0$ are non-vanishing. Next, we utilize the fact that the sum of residues of a rational function 
at all finite points equals its residue at infinity, which is the coefficient at $z^{-1}$ in the asymptotic expansion 
\[
F_k(z)\sim \sum_{j =-\infty}^{M+ v-N-r-t} C_j(k) z^j~~\text{as}~~z\to\infty.
\]
Therefore, we have 
\begin{equation}\label{eq:residues-coefficients}
\sum_{i = 1}^r \sum_{j = 0}^{k + n_i} \gamma_{i, j}^{k + n_i} = C_{-1}(k) - \res_{z = 0} F_k(z).    
\end{equation}
Corollary~\ref{COR:qBeukersJouhet2} implies that the coefficient $C_{-1}(k)$ at $z^{-1}$ equals $0$ if $M + v - N - r - t < -1$, and  equals  $(q^vB)^k Q_p(q^{-k})$ if $M + v - N - r - t = p - 1$ with $p \in \N_{0}$, where $Q_p(y) = \sum_{i = 0}^p a_{p, i} y^i$ is a polynomial of degree $p$ in $y$. Similarly, the residue of $F_k(z)$ at $z = 0$ is equal to $0$ if $u + t \leq 0$, and is equal to $\tilde{Q}_{u + t - 1}(q^k)$ if $u + t \geq 1$, where $\tilde{Q}_{u + t - 1}(y) = \sum_{i = 0}^{u + t - 1} b_{u + t - 1, i} y^i$ is a polynomial of degree $u + t - 1$ in $y$. Furthermore, we extend the definitions of $Q_{p}$ and $\tilde{Q}_{u+t-1}$ to negative degrees by $Q_{-\ell}(y) = \tilde{Q}_{-\ell}(y) \equiv 0$ for each $\ell = 1, 2, \ldots$. 

Set $$
k_{\min}=\min(n_{\max},m_{\min}).
$$ 
In view of \eqref{EQ:qtargetsum2}, the target sum~\eqref{EQ:targetsum3} equals (the first sum vanishes if $-n_{max}\ge-m_{min}$)
\begin{align*} 
S(z) & = \sum_{k = - n_{\max}}^{- m_{\min} - 1} \mu_k z^k + \sum_{k = - k_{\min}}^{\infty} z^k \sum_{i = 1}^r \sum_{j = 0}^{k + n_i} \gamma_{i, j}^{k + n_i} \nonumber \\
& = \sum_{k = - n_{\max}}^{- m_{\min} - 1} \mu_k z^k +  \sum_{k = - k_{\min}}^{\infty} z^k \left[ C_{-1}(k) - \res_{z = 0} F_k(z) \right] \nonumber \\
& = \sum_{k = - n_{\max}}^{- m_{\min} - 1} \mu_k z^k +  \sum_{k = - k_{\min}}^{\infty} z^k \left[(q^{v}B)^k Q_p(q^{-k}) - \tilde{Q}_{u + t - 1}(q^k) \right]\nonumber  \\
& = \sum_{k = - n_{\max}}^{- m_{\min} - 1} \mu_k z^k +  \sum_{k = - k_{\min}}^{\infty} (q^{v}B z)^k \left[a_{p, 0} + a_{p, 1} q^{-k} + \cdots + a_{p, p} q^{-p k} \right] \nonumber \\
& \quad \ - \sum_{k = - k_{\min}}^{\infty} z^k \left[b_{u + t -1, 0} + b_{u + t - 1, 1} q^k + \cdots + b_{u + t - 1, u + t - 1} q^{k (u + t - 1)}\right] \nonumber \\ 
& =  \sum_{k = - n_{\max}}^{- m_{\min} - 1} \mu_k z^k  + z^{- k_{\min}} \left[ \frac{a_{p, 0} (q^{v}B)^{-k_{\min}}}{1 - q^{v}Bz} + \frac{a_{p, 1} (q^{v}B)^{- k_{\min}} q^{k_{\min}}}{1 - q^{v-1}Bz} + \cdots + \frac{a_{p,p} (q^{v}B)^{-k_{\min}} q^{p k_{\min}}}{1 - q^{v-p}Bz} \right. \\ 
& \quad \ \left. - \frac{b_{u + t - 1, 0}}{1 - z} -  \frac{b_{u + t - 1, 1}q^{-k_{\min}}}{1 - qz} - \cdots -  \frac{b_{u + t - 1, u + t - 1}q^{-k_{\min}(u+t-1)}}{1 - q^{u + t - 1}z} \right].
\end{align*}
Note that the common denominator on the right-hand side of the above identity equals $(W z; q)_{p + 1} (z; q)_{(u + t)_{+}}$ and thus we get~\eqref{EQ:generalizedqBeukersJouhet}. 

Next, we deduce an explicit formula for the coefficient $\lambda_k$ in~\eqref{EQ:generalizedqBeukersJouhet}. We first recall the following two well-known identities 
\begin{align} \label{EQ:qPochhammerformula2}
(q; q)_{k - j} & = \frac{(-1)^j q^{j (j - 1)/2 - kj} (q; q)_k}{(q^{-k}; q)_j}, \nonumber \\
(q^{s - j}; q)_m & = \frac{(q^s; q)_m (q^{1 - s}; q)_j}{q^{m j} (q^{1 - s - m}; q)_j} .
\end{align}
Applying the above two identities to~\eqref{EQ:defgamma3}, we get 
\begin{align*}
\gamma_{i, j}^{k + n_i} & = \frac{(q^{- k - n_i}; q)_j (q^{\mathbf{b} - a_i}; q)_j (q^{\mathbf{a}_{[i]} - a_i - \mathbf{n}_{[i]} - k}; q)_j(q^{1 - \mathbf{b} + a_i}; q)_{\mathbf{m} + k}}{(q; q)_{k + n_i} (q^{a_i - \mathbf{a}_{[i]}}; q)_{\mathbf{n}_{[i]} + k + 1} (q^{\mathbf{b} - a_i - \mathbf{m} - k}; q)_j (q^{1 - a_i + \mathbf{a}_{[i]}}; q)_j (q; q)_j} 
\\
& \quad \ \times \frac{q^{a_i(1 - t)} q^{(N - M + r -1+t-v)j}  (q^{- a_i + \boldsymbol{\alpha}}; q)_j(q^{-a_i+ \boldsymbol{\alpha} -1}; q)_1 (q^{-a_i - \boldsymbol{\beta} - k}; q)_j (q^{1 + a_i + \boldsymbol{\beta} + k}; q)_1 }{(q^{- 1 - a_i - \boldsymbol{\beta} - k}; q)_j (q^{- a_i + \boldsymbol{\alpha} - 1}; q)_j}.
\end{align*}
Thus, from \eqref{eq:Sintermsofgammas} we have $S(z) = \sum_{k=-n_{\max}}^\infty \mu_kz^k$ with
\begin{align} \label{EQ:muexplicitformula}
\mu_k & = \sum_{i = 1}^r \sum_{j = 0}^{k + n_i} \gamma_{i, j}^{k + n_i}  \nonumber \\
& = \sum_{i = 1}^r \frac{q^{a_i(1-t)}  (q^{1 + a_i + \boldsymbol{\beta} + k}; q)_1 (q^{-a_i+\boldsymbol{\alpha} -1}; q)_1 (q^{1 - \mathbf{b} + a_i}; q)_{\mathbf{m} + k}}{(q; q)_{k + n_i} (q^{a_i - \mathbf{a}_{[i]}}; q)_{\mathbf{n}_{[i]} + k + 1}}  \nonumber \\
& \quad \ \times {_{2 r + u + v}\phi_{2 r + u + v - 1}}\left(\left.\!\!\begin{array}{c} q^{- k - n_i}, q^{\mathbf{b} - a_i}, q^{\mathbf{a}_{[i]} - a_i - \mathbf{n}_{[i]} - k}, q^{- a_i - \boldsymbol{\beta} - k}, q^{-a_i + \boldsymbol{\alpha}} \\ q^{\mathbf{b} - a_i - \mathbf{m} - k}, q^{1 - a_i + \mathbf{a}_{[i]}}, q^{-1 - a_i - \boldsymbol{\beta} -  k}, q^{- a_i + \boldsymbol{\alpha} - 1} \end{array} \right| q^{N - M + r-1 + t-v} \!\right),
\end{align}
where each term with $k + n_i < 0$ vanishes by convention. Next, we need the Gauss expansion 
\[
(a; q)_n = \sum_{j = 0}^n \left[ \begin{matrix} n \\ j \end{matrix} \right]_{q} q^{j (j - 1)/2} (-a)^j,
\]
where $\left[ \begin{matrix} n \\ j \end{matrix} \right]_{q}$ is the $q$-binomial coefficient given in \eqref{EQ:Gaussexpansion}. This leads to  
\[
(Wz; q)_{p + 1} (z; q)_{(u + t)_{+}} := \sum_{j = 0}^{p + 1 + (u + t)_{+}} D_j (-z)^j,
\]
where 
\[
D_j = \sum_{i + \ell = j} \left[ \begin{matrix} p + 1 \\ \ell \end{matrix} \right]_{q} \left[ \begin{matrix} (u + t)_{+} \\ i \end{matrix} \right]_{q} q^{(\ell (\ell - 1) + i (i - 1))/2} W^\ell.
\]
Note that for $u+t\le0$ we have
$$
D_j=\left[ \begin{matrix} p + 1 \\ j \end{matrix} \right]_{q}  q^{j(j - 1)/2} W^j.
$$
Multiplying both sides of~\eqref{EQ:generalizedqBeukersJouhet} by $(Wz;q)_{p+1}(z;q)_{(u + t)_{+}}$ and applying the aforementioned expansion, we get 
\[
\sum_{j = 0}^{p + 1 + (u + t)_{+}} D_j (- z)^j \sum_{k = - n_{\max}}^\infty \mu_k z^k = \sum_{k = - n_{\max}}^{p + (u + t)_{+} - m_{\min}} \lambda_k z^k. 
\]
Multiplying both sides by $z^{n_{\max}}$, changing $k+n_{\max}\to k$ and writing $\hat{\mu}_k = \mu_{k - n_{\max}}, \hat{\lambda}_k = \lambda_{k - n_{\max}}$, we obtain 
\[
\sum_{j = 0}^{p + 1 + (u + t)_{+}} D_j (- z)^j \sum_{k = 0}^\infty \hat{\mu}_k z^k = \sum_{s = 0}^{\infty}z^{s}\sum_{j + k = s} D_j (-1)^j \hat{\mu}_k \:=\!\! \sum_{s = 0}^{p + (u + t)_{+} - m_{\min} + n_{\max}} \hat{\lambda}_s z^s.
\]
On account of $D_j = 0$ for $j > p + (u + t)_{+} + 1$, it implies that 
\[
\lambda_{s - n_{\max}} = \hat{\lambda}_s = \sum_{j = 0}^{\min(s, p + (u + t)_{+} + 1)} D_j (-1)^j \hat{\mu}_{s - j} = \sum_{j = 0}^{\min(s, p + (u + t)_{+} + 1)} D_j (-1)^j \mu_{s - j - n_{\max}}
\]
for $s = 0, \ldots, p + (u + t)_{+} - m_{\min} + n_{\max}$. Set $k = s - n_{\max}$. Then 
\[
\lambda_k = \sum_{j = 0}^{\min(k + n_{\max}, p + (u + t)_{+} + 1)} (-1)^j D_j \mu_{k - j} = \sum_{j = \max(- n_{\max}, k - p - (u + t)_{+} - 1)}^k (-1)^{k - j} D_{k - j} \mu_j
\]
for $k = - n_{\max}, \ldots, p + (u + t)_{+} - m_{\min}$. Substituting the formula~\eqref{EQ:muexplicitformula} for $\mu_k$ in the above identity, we finally arrive at the explicit formula~\eqref{EQ:explicitformulalambda} for $\lambda_k$.
\hfill
\qedsymbol

\begin{remark}
The identity established in \cite[Theorem~1]{KKK} is recovered from \eqref{EQ:generalizedqBeukersJouhet} by setting $u = v = 0$.
\end{remark}

\begin{remark}
The $q$-identity of Beukers-Jouhet  \cite[Theorem 1.2]{BJ} is derived from~\eqref{EQ:generalizedqBeukersJouhet} by taking $\mathbf{m}$, $\mathbf{n}$, $u$, and $v$ to be zero vectors of appropriate sizes and setting $z = q^k \bar{z}$ and $t = \ell - k$ 
for some $k,\ell\in\N$. 
\end{remark}

\noindent Below we present an example of the right-hand side of \eqref{EQ:generalizedqBeukersJouhet} with $r = 3$, $u=v=1$. 
\begin{example} \label{EX: qex1}
Set $\mathbf{m} = (1,1,2), \mathbf{n} = (1,2,2),  \mathbf{a} = (1, 1/2, 1/3), \mathbf{b} = (2, 3, 4), \boldsymbol{\alpha} = (1/2)$, and $\boldsymbol{\beta} = (1/3)$. Then the right-hand side of~\eqref{EQ:generalizedqBeukersJouhet} takes the form 
\begin{multline*}
\frac{1}{z^2 (1 - z)} \left[ \left(-\frac{\left(1-q^{-13/12}\right) q^{1/3} \left(1-q^{-7/15}\right)}{\left(1-q^{8/3}\right) \left(1-q^{5/3}\right)
   \left(1-q^{-1/6}\right)}-\frac{\left(1-q^{-5/4}\right) \left(1-q^{-3/10}\right) q^{1/2}}{\left(1-q^{-5/2}\right)
   \left(1-q^{-3/2}\right) \left(1-q^{1/6}\right)}-q^{-3/4}\right) z \right. \\
   \left. +\frac{\left(1-q^{-13/12}\right) q^{1/3} \left(1-q^{-7/15}\right)}{\left(1-q^{8/3}\right) \left(1-q^{5/3}\right)
   \left(1-q^{-1/6}\right)}+\frac{\left(1-q^{-5/4}\right) \left(1-q^{-3/10}\right) q^{1/2}}{\left(1-q^{-5/2}\right)
   \left(1-q^{-3/2}\right) \left(1-q^{1/6}\right)} \right].
\end{multline*}
\end{example}

The following corollary is established by employing the same argument as that appears in the proof of \cite[Proposition~1]{KKK}.
\begin{corollary}
Suppose conditions of Theorem~\ref{THM: qBeukersJouhet2} are satisfied and $|W|<1$. Then
\begin{multline*}
    \sum_{i = 1}^r q^{a_i(1 - t)} \frac{(q^{1 - \mathbf{b} + a_i}; q)_{\infty} (q^{- a_i - 1 + \boldsymbol{\alpha}}; q)_1}{(q^{a_i - \mathbf{a}_{[i]}}; q)_{\infty}} {_{r + u}\phi_{r + u - 1}}\!\left(\left.\!\!\begin{array}{c} q^{\mathbf{b} - a_i}, q^{- a_i + \boldsymbol{\alpha}} \\ q^{1 + \mathbf{a}_{[i]} - a_i}, q^{- a_i - 1 + \boldsymbol{\alpha}} \end{array} \right| W\!\right)  \\
 = \frac{(q^{u+t}; q)_{\infty}}{(W; q)_{p+1}} \sum_{k=-n_{\max}}^{p+(u+t)_{+}-m_{\min}}\lambda_k,
\end{multline*}
where $\lambda_k$ is specified by~\eqref{EQ:explicitformulalambda}. Note that $(q^{u+t}; q)_{\infty}=0$ if $u+t\le0$, so that the right-hand side vanishes under this condition. 
\end{corollary}

Our next goal is to deduce an identity for terminating $q$-series. 
Denote 
\begin{equation}\label{eq:ABM2N2}
A=\sum a_j,~~B=\sum b_j,~~M_2=\sum m_i^2,~~N_2=\sum n_i^2,~~\alpha=\sum\alpha_j, ~~\beta=\sum\beta_j   
\end{equation}
We will need an alternative asymptotic formula for $\tilde{f}_k(z)$ defined in 
\eqref{eq:fF-defined}. In view of 
$$
(a;q)_{n}=(-a)^nq^{\binom{n}{2}}\left(\frac{q^{1-n}}{a};q\right)_{n},
$$
we have by repeating \emph{mutatis mutandis} the calculation from \cite[p.404]{KKK}
\begin{multline*}
\tilde{f}_k(z)= - \frac{(zq^{1 - \mathbf{b}}; q)_{k + \mathbf{m}} (z q^{1 + \boldsymbol{\beta} + k}; q)_1 (z^{-1} q^{\boldsymbol{\alpha} - 1}; q)_1}{(z q^{- \mathbf{a}}; q)_{k + \mathbf{n}+1}}
\\
=(-1)^{M-N-r+v+1}q^{k(A-B+M-N+v)}q^{A+(\a,\n)-(\b,\m)+(M_2-N_2+M-N)/2+v+\beta}
\\
\times z^{M-N-r+v} 
\times\underbrace{\frac{\left(z^{-1}q^{\b-k-\m};q\right)_{k+\m}(z^{-1} q^{\boldsymbol{\alpha} - 1}; q)_1 (z^{-1}q^{-\boldsymbol{\beta}-k-1}; q)_1}{\left(z^{-1}q^{\a-k-\n};q\right)_{k+\n+1}}}_{\hat{f}_k(z)}.
\end{multline*}
Then, by the standard Taylor series  $\log(1-x)=-\sum_{s\ge1}x^s/s$, we compute
\begin{multline*}
\log\hat{f}_k(z)=\sum_{i=1}^{r}\Bigg[\sum_{j=0}^{k+m_i-1}\log(1-z^{-1}q^{b_i-m_i-k+j})-\sum_{j=0}^{k+n_i}\log(1-z^{-1}q^{a_i-n_i-k+j})\Bigg]
\\
+\sum_{\ell=1}^{u}\log(1-z^{-1}q^{\alpha_{\ell}-1})
+\sum_{\ell=1}^{v}\log(1-z^{-1}q^{-\beta_{\ell}-k-1})
\\
=\sum_{i=1}^{r}\Bigg[\sum_{j=0}^{k+n_i}\sum_{s\ge1}\frac{1}{s}z^{-s}q^{(a_i-n_i-k+j)s}-\sum_{j=0}^{k+m_i-1}\sum_{s\ge1}\frac{1}{s}z^{-s}q^{(b_i-m_i-k+j)s}\Bigg]
\\
-\sum_{\ell=1}^{u}\sum_{s\ge1}\frac{z^{-s}}
{s}q^{s(\alpha_{\ell}-1)}
-\sum_{\ell=1}^{v}\sum_{s\ge1}\frac{z^{-s}}
{s}q^{-s(\beta_{\ell}+k+1)}
=\sum_{s\ge1}z^{-s}Q_s,
\end{multline*}
where
\begin{subequations}\label{eq:Qq}
\begin{multline}\label{eq:BigQ}
Q_s=Q_s(\a,\b,\m,\n,\boldsymbol{\alpha},\boldsymbol{\beta})=\frac{1}{s}\sum_{i=1}^{r}\Bigg[\sum_{j=0}^{k+n_i}q^{(a_i-n_i-k+j)s}-\sum_{j=0}^{k+m_i-1}q^{(b_i-m_i-k+j)s}\Bigg]
\\
-\frac{1}{s}\sum_{\ell=1}^{u}q^{s(\alpha_{\ell}-1)}
-\frac{1}{s}\sum_{\ell=1}^{v}q^{-s(\beta_{\ell}+k+1)}
.
\end{multline}
Then
$$
\hat{f}_k(z)=\sum\limits_{s=0}^{\infty}\frac{\hat{q}_s}{z^s}
$$
with
\begin{equation}\label{eq:smnallq}
\hat{q}_0=1,~~~~\hat{q}_s=\frac{1}{s}\sum_{t=1}^{s}tQ_{t}\hat{q}_{s-t}.
\end{equation}
\end{subequations}
Hence,
$$
F_{k}(z)=(-1)^{M-N-r+v}q^{k(A-B+M-N+v)}q^{A+\a\cdot\n-\b\cdot\m+(M_2-N_2+M-N)/2+v+\beta}z^{M-N-r+v-t}\sum\limits_{s=0}^{\infty}\frac{\hat{q}_s}{z^s},
$$
where $\a\cdot\n=a_1n_1+\cdots +a_rn_r$.  This implies that the coefficient $C_{-1}(k)$ in the expansion (for large $z$)
$$
F_{k}(z)=\sum\limits_{j=-\infty}^{M-N-r+v-t}C_{j}(k)z^{j}.
$$
is equal to zero if $M-N-r+v-t<-1$ and  given by 
$$
C_{-1}(k)=(-1)^{M-N-r+v}q^{k(A-B+M-N+v)}q^{A+\a\cdot\n-\b\cdot\m+(M_2-N_2+M-N)/2+v+\beta}\hat{q}_p
$$
if $p=M-N-r+v-t+1\ge0$, and $C_{-1}(k)=0$ if $p=M-N-r+v-t+1<0$. Hence, defining $\hat{q}_{s}=0$ for $s<0$ and using our previous definition $p=\max(-1,M-N-r+v-t+1)$ we get consistent definition of $\hat{q}_p$ for all values of parameters.

In a similar fashion, when $z$ is small enough, we use the notation from \eqref{eq:ABM2N2} to get
$$
\tilde{f}_k(z)= (-1)^{u+1}z^{-u}q^{\alpha-u}\: \underbrace{\frac{(zq^{1 - \mathbf{b}}; q)_{k + \mathbf{m}} (zq^{1 + \boldsymbol{\beta} + k}; q)_1 (zq^{1-\boldsymbol{\alpha}}; q)_1}{(z q^{- \mathbf{a}}; q)_{k + \mathbf{n}+1}}}_{\bar{f}_{k}(z)},
$$
so that
\begin{multline*}
\log(\bar{f}_{k}(z))=\sum_{i=1}^{r}\Bigg[\sum_{j=0}^{k+m_i-1}\log(1-zq^{1-b_i+j})-\sum_{j=0}^{k+n_i}\log(1-zq^{-a_i+j})\Bigg]
\\
+\sum_{\ell=1}^{u}\log(1-zq^{1-\alpha_{\ell}})
+\sum_{\ell=1}^{v}\log(1-zq^{1+\beta_{\ell}+k})
=\sum_{s=1}^{\infty}P_{s}z^{s},
\end{multline*}
where
\begin{subequations}\label{eq:Pp}
\begin{equation}\label{eq:bigp}
sP_s=sP_s(\a,\b,\m,\n,\boldsymbol{\alpha},\boldsymbol{\beta})=\sum_{i=1}^{r}\Bigg[\sum_{j=0}^{k+n_i}q^{(-a_i+j)s}-\sum_{j=0}^{k+m_i-1}q^{(1-b_i+j)s}\Bigg]-\sum_{\ell=1}^{u}q^{(1-\alpha_{\ell})s}-\sum_{\ell=1}^{v}q^{(1+\beta_{\ell}+k)s}.
\end{equation}
Hence,
$$
F_{k}(z)=z^{-t}\tilde{f}_k(z)=(-1)^{u+1}q^{\alpha-u}z^{-u-t}e^{\log(\bar{f}_k(z))}=(-1)^{u+1}q^{\alpha-u}z^{-u-t}\sum_{s=0}^{\infty}p_{s}z^s
$$
with
\begin{equation}\label{eq:smnallp}
p_0=1,~~~~p_s=\frac{1}{s}\sum_{t=1}^{s}tP_{t}p_{s-t}.
\end{equation}
\end{subequations}
This yields for the coefficient at $z^{-1}$
$$
\res\limits_{z=0}{F_k(z)}=(-1)^{u+1}q^{\alpha-u}p_{u+t-1}
$$
provided that $u+t\ge1$. Otherwise, the residue is equal to zero. So, we set by definition $p_{s}=0$ for $s<0$.
On the other hand, it follows from \eqref{eq:residues-coefficients} and \eqref{eq:Sintermsofgammas} 
that for $k\ge-m_{\min}$ 
$$
C_{-1}(k)- \res\limits_{z=0}{F_k(z)}=\sum_{i=1}^{r}\sum_{j=0}^{k+n_i}\gamma_{i,j}^{k+n_i}=\mu_k.
$$
Substituting here the values of $\mu_k$ from \eqref{EQ:muexplicitformula}, we arrive at
\begin{corollary}
Suppose  that $k\ge-m_{\min}$, $\hat{q}_{-1}=0$, $\hat{q}_0=1$, and  $\hat{q}_p=\hat{q}_p(\a,\b,\m,\n,\boldsymbol{\alpha},\boldsymbol{\beta})$ is given by \eqref{eq:Qq} for $p\ge1$.  
Similarly, let $p_{-1}=0$, $p_0=1$, and $p_{s}=p_{s}(\a,\b,\m,\n,\boldsymbol{\alpha},\boldsymbol{\beta})$ be given by \eqref{eq:Pp} for $s\ge1$.
Then 
\begin{multline}\label{eq:alphaIdentity}
\sum_{i = 1}^r \frac{q^{a_i(1-t)}  (q^{1 + a_i + \boldsymbol{\beta} + k}; q)_1 (q^{-a_i+\boldsymbol{\alpha} -1}; q)_1 (q^{1 - \mathbf{b} + a_i}; q)_{\mathbf{m} + k}}{(q; q)_{k + n_i} (q^{a_i - \mathbf{a}_{[i]}}; q)_{\mathbf{n}_{[i]} + k + 1}}  \nonumber \\
\times {_{2 r + u + v}\phi_{2 r + u + v - 1}}\left(\left.\!\!\begin{array}{c} q^{- k - n_i}, q^{\mathbf{b} - a_i}, q^{\mathbf{a}_{[i]} - a_i - \mathbf{n}_{[i]} - k}, q^{- a_i - \boldsymbol{\beta} - k}, q^{-a_i + \boldsymbol{\alpha}} \\ q^{\mathbf{b} - a_i - \mathbf{m} - k}, q^{1 - a_i + \mathbf{a}_{[i]}}, q^{-1 - a_i - \boldsymbol{\beta} -  k}, q^{- a_i + \boldsymbol{\alpha} - 1} \end{array} \right| q^{N - M + r-1 + t-v} \!\right)
\\
=(-1)^{M-N-r+v}q^{k(A-B+M-N+v)}q^{A+ \a\cdot\n - \b\cdot\m+(M_2-N_2+M-N)/2+v+\beta}\hat{q}_p-(-1)^{u+1}q^{\alpha-u}p_{u+t-1},
\end{multline}
where $A,B,M_2,N_2,\alpha,\beta$ are defined by \eqref{eq:ABM2N2} and $\a\cdot\n=a_1n_1+\cdots +a_rn_r$ is the scalar product. 
\end{corollary}
 
\medskip

To formulate the  confluent case of identity \eqref{EQ:generalizedqBeukersJouhet}, we recall the confluent case of the  basic hypergeometric function  \cite[formula (1.2.22)]{GR04}:
\begin{equation} \label{EQ:basicqhyper}
{_{s}\phi_{r}}\left(\left.\!\!\begin{array}{c} \mathbf{a} \\ \mathbf{b} \end{array} \right| q, z\!\right) 
= \sum_{n = 0}^\infty \frac{(a_1; q)_n (a_2; q)_n \cdots (a_s; q)_n}{(b_1; q)_n (b_2; q)_n \cdots (b_r;q)_n (q; q)_n} 
\left[(-1)^n q^{\binom{n}{2}} \right]^{1 + r - s} z^n,
\end{equation}
where $s \leq r$, and the series is convergent for all complex $z$. 
It is further convenient to introduce another version of the basic hypergeometric series considered by Bailey~\cite[Sect. 2]{Bailey64} and Slater~\cite[(3.2.1.11)]{Slater} and given by 
\begin{equation} \label{EQ:basicqhyper2}
{_{s}\hat{\phi}_{r}}\left(\left.\!\!\begin{array}{c} \mathbf{a} \\ \mathbf{b} \end{array} \right| q, z\!\right) 
= \sum_{n = 0}^\infty \frac{(a_1; q)_n (a_2; q)_n \cdots (a_s; q)_n}{(b_1; q)_n (b_2; q)_n \cdots (b_r; q)_n (q; q)_n} 
 z^n,
\end{equation}
Note that for $s = r + 1$ the two series ${_{s}\phi_{r}}$ and ${_{s}\hat{\phi}_{r}}$ coincide. If $s<r+1$ the Bailey-Slater series~\eqref{EQ:basicqhyper2} can be expressed in terms of the series \eqref{EQ:basicqhyper} by setting some of the parameters equal to zero:
\begin{equation} \label{EQ:basicqhyper2to1}
{_{s}\hat{\phi}_{r}}\left(\left.\!\!\begin{array}{c} a_1, \ldots, a_s \\ b_1, \ldots, b_r \end{array} \right| q, z\!\right) = 
 {_{r+1}\phi_{r}}\left(\left.\!\!\begin{array}{c} a_1, \ldots, a_s, 0, \ldots, 0 \\ b_1, \ldots, b_r \end{array} \right| q, z\!\right). 
\end{equation}

\begin{theorem} \label{THM: qBeukersJouhet3}
Suppose $0 \leq s < r$  are integers, $q$ with $0 < |q| < 1$ is a complex number, the vector $\mathbf{a} \in \C^r$ satisfies $a_i - a_j \not\in \Z$ for $1 \leq i < j \leq r$, and the vector $\mathbf{b} \in \C^s$  is arbitrary. Assume further that $\mathbf{m} \in \Z^s, \mathbf{n} \in \Z^r$ and $t \in \Z$. Let $\boldsymbol{\alpha} \in \C^u$, $\boldsymbol{\beta} \in \C^v$, and  $W = q^{t + r -1} \prod_{i = 1}^{r} q^{a_i} \prod_{i = 1}^{s}q^{ - b_i}$.  Then
\begin{align} 
&\sum_{i = 1}^r q^{a_i (1 - t)} \frac{(q^{1 - \mathbf{b} + a_i}; q)_{\mathbf{m} - n_i} (q^{- a_i - 1 + \boldsymbol{\alpha}}; q)_1 (q^{1 + a_i + \boldsymbol{\beta}-n_i}; q)_1 z^{-n_i}}{(q^{a_i - \mathbf{a}_{[i]}}; q)_{\mathbf{n}_{[i]} - n_i + 1}} \nonumber \\
&  \times{_{s + u}\phi_{r + u - 1}}\left(\left.\!\!\begin{array}{c} q^{\mathbf{b} - a_i}, q^{- a_i + \boldsymbol{\alpha}} \\ q^{1 + \mathbf{a}_{[i]} - a_i}, q^{- a_i - 1 + \boldsymbol{\alpha}} \end{array} \right| W q^{(s - r) a_i} z\!\right) \times {_{s + v}\hat{\phi}_{r + v - 1}}\left(\left.\!\!\begin{array}{c} q^{1 - \mathbf{b} + a_i + \mathbf{m} - n_i}, q^{2 + a_i + \boldsymbol{\beta}-n_i} \nonumber\\ q^{1 - \mathbf{a}_{[i]} + a_i + \mathbf{n}_{[i]} - n_i}, q^{1 + a_i + \boldsymbol{\beta}-n_i} \end{array} \right| z\!\right) \\ 
& = \frac{1}{(z; q)_{(u + t)_{+}}} \sum_{k = - n_{\max}}^{K} \delta_k z^k, \label{EQ:generalizedqBeukersJouhet2}
\end{align}
\noindent
where $p' = \floor{(M + v - N - r -t + 1)/(r - s)}$ and $K = \max(-m_{\min} - 1, p') + (u + t)_{+}$. The symbols $M$, $N$, $n_{\max}$, $m_{\min}$ retain their meaning from \eqref{eq:mnconfluent}.
The coefficient $\delta_k$ is given by
\begin{multline}\label{eq:qdelta_k}
\delta_{k}=
\sum_{j=\max(-n_{\max},k-(u+t)_{+})}^{k} \left[ \begin{matrix} (u+t)_{+} \\ k-j \end{matrix} \right]_{q} q^{(k-j) (k-j -1)/2} (-1)^{k-j} 
\\
\times\sum_{i = 1}^r \frac{q^{a_i(1-t)}  (q^{1 + a_i + \boldsymbol{\beta} + j}; q)_1 (q^{-a_i+\boldsymbol{\alpha} -1}; q)_1 (q^{1 - \mathbf{b} + a_i}; q)_{\mathbf{m} + j}}{(q; q)_{j + n_i} (q^{a_i - \mathbf{a}_{[i]}}; q)_{\mathbf{n}_{[i]} + j + 1}}
\\
\times{_{r+s + u + v}\phi_{r+s + u + v - 1}}\left(\left.\!\!\begin{array}{c} q^{- j - n_i}, q^{\mathbf{b} - a_i}, q^{\mathbf{a}_{[i]} - a_i - \mathbf{n}_{[i]} - j}, q^{- a_i - \boldsymbol{\beta} - j}, q^{-a_i + \boldsymbol{\alpha}} \\ q^{\mathbf{b} - a_i - \mathbf{m} - j}, q^{1 - a_i + \mathbf{a}_{[i]}}, q^{-1 - a_i - \boldsymbol{\beta} -  j}, q^{- a_i + \boldsymbol{\alpha} - 1} \end{array} \right| q^{N - M + r-1 + t-v} \!\right)    
\end{multline}
for $k=-n_{\max},\ldots, K$.
\end{theorem}

\textbf{Remark.}  If $u+t\le0$, then \eqref{eq:qdelta_k} simplifies to 
\begin{multline*}
\delta_{k}=
\sum_{i = 1}^r \frac{q^{a_i(1-t)}  (q^{1 + a_i + \boldsymbol{\beta} + k}; q)_1 (q^{-a_i+\boldsymbol{\alpha} -1}; q)_1 (q^{1 - \mathbf{b} + a_i}; q)_{\mathbf{m} + k}}{(q; q)_{k + n_i} (q^{a_i - \mathbf{a}_{[i]}}; q)_{\mathbf{n}_{[i]} + k + 1}}
\\
\times{_{r+s + u + v}\phi_{r+s + u + v - 1}}\left(\left.\!\!\begin{array}{c} q^{- k - n_i}, q^{\mathbf{b} - a_i}, q^{\mathbf{a}_{[i]} - a_i - \mathbf{n}_{[i]} - k}, q^{- a_i - \boldsymbol{\beta} - k}, q^{-a_i + \boldsymbol{\alpha}} \\ q^{\mathbf{b} - a_i - \mathbf{m} - k}, q^{1 - a_i + \mathbf{a}_{[i]}}, q^{-1 - a_i - \boldsymbol{\beta} -  k}, q^{- a_i + \boldsymbol{\alpha} - 1} \end{array} \right| q^{N - M + r-1 + t-v} \!\right).    
\end{multline*}

\noindent
\textit{Proof of Theorem~\ref{THM: qBeukersJouhet3}.} Repeating the proof of Theorem~\ref{THM: qBeukersJouhet2}, we calculate
\begin{align*} 
S(z) & := \sum_{i = 1}^r q^{a_i (1 - t)} \frac{(q^{1 - \mathbf{b} + a_i}; q)_{\mathbf{m} - n_i} (q^{- a_i - 1 + \boldsymbol{\alpha}}; q)_1 (q^{1 + a_i + \boldsymbol{\beta}-n_i}; q)_1 z^{-n_i}}{(q^{a_i - \mathbf{a}_{[i]}}; q)_{\mathbf{n}_{[i]} - n_i + 1}}  \\
& \ \quad \times{_{s + u}\phi_{r + u - 1}}\left(\left.\!\!\begin{array}{c} q^{\mathbf{b} - a_i}, q^{- a_i + \boldsymbol{\alpha}} \\ q^{1 + \mathbf{a}_{[i]} - a_i}, q^{- a_i - 1 + \boldsymbol{\alpha}} \end{array} \right| W q^{(s - r) a_i} z\!\right) \times {_{s + v}\hat{\phi}_{r + v - 1}}\left(\left.\!\!\begin{array}{c} q^{1 - \mathbf{b} + a_i + \mathbf{m} - n_i}, q^{2 + a_i + \boldsymbol{\beta}-n_i} \\ q^{1 - \mathbf{a}_{[i]} + a_i + \mathbf{n}_{[i]} - n_i}, q^{1 + a_i + \boldsymbol{\beta}-n_i} \end{array} \right| z\!\right)  \\
& = \sum_{i = 1}^r \sum_{k = 0}^\infty z^{k - n_i} \\
& \ \quad \times\sum_{j = 0}^k \frac{q^{a_i (1 - t)} (q^{1 - \mathbf{b} + a_i}; q)_{\mathbf{m} - n_i} (q^{\mathbf{b} - a_i}; q)_j (q^{1 - \mathbf{b} + a_i + \mathbf{m} - n_i}; q)_{k - j} (q^{- a_i + \boldsymbol{\alpha} + j - 1}; q)_1 (q^{1 + a_i + \boldsymbol{\beta}-n_i + k - j}; q)_1 }{(q^{a_i - \mathbf{a}_{[i]}}; q)_{\mathbf{n}_{[i]} - n_i + 1}(q^{1 + \mathbf{a}_{[i]} - a_i}; q)_j (q^{1 - \mathbf{a}_{[i]} + a_i + \mathbf{n}_{[i]} - n_i}; q)_{k - j} (q; q)_j (q; q)_{k - j}} \\
& \quad \ \times (W q^{(s - r) a_i})^j\left[(-1)^j q^{j (j - 1)/2} \right]^{r - s} \\
& = \sum_{i = 1}^r \sum_{k = 0}^\infty z^{k - n_i} \sum_{j = 0}^k  \gamma_{i, j}^k = \sum_{i = 1}^r \sum_{k_i = - n_i}^\infty z^{k_i} \sum_{j = 0}^{k_i + n_i} \gamma_{i, j}^{k_i + n_i}, 
\end{align*}
where 
\begin{align} \label{EQ:defgamma4}
\gamma_{i, j}^k  & =  \frac{q^{a_i (1 - t)} (q^{1 - \mathbf{b} + a_i}; q)_{\mathbf{m} - n_i} (q^{\mathbf{b} - a_i}; q)_j (q^{1 - \mathbf{b} + a_i + \mathbf{m} - n_i}; q)_{k - j} (q^{- a_i + \boldsymbol{\alpha} + j - 1}; q)_1 (q^{1 + a_i + \boldsymbol{\beta}-n_i + k - j}; q)_1}{(q^{a_i - \mathbf{a}_{[i]}}; q)_{\mathbf{n}_{[i]} - n_i + 1}(q^{1 + \mathbf{a}_{[i]} - a_i}; q)_j (q^{1 - \mathbf{a}_{[i]} + a_i + \mathbf{n}_{[i]} - n_i}; q)_{k - j} (q; q)_j (q; q)_{k - j}} \nonumber \\
& \quad \ \times  (W q^{(s - r) a_i})^j\left[(-1)^j q^{j (j - 1)/2} \right]^{r - s}. 
\end{align}
Using a similar computation as that in the proof of Theorem~\ref{THM: qBeukersJouhet2}, we deduce that 
\begin{align*}
\gamma_{i, j}^k & = \frac{(-1)^{(s - r  + 1) j} q^{(s - r + 1)j (j - 1)/2 + a_{i}(1 - t) + t j} (q^{1 - \mathbf{b} + a_i - j}; q)_{\mathbf{m} + k - n_i} (q^{- a_i + \boldsymbol{\alpha} + j - 1}; q)_1 (q^{1 + a_i + \boldsymbol{\beta}-n_i + k - j}; q)_1}{(q^{a_i - \mathbf{a}_{\mathbf{[i]}} - j}; q)_{\mathbf{n}_{[i]} + k - n_i + 1} (q; q)_j (q; q)_{k - j}} \\
& \quad \  \times \left[(-1)^j q^{j (j - 1)/2} \right]^{r - s} \\[6pt]
& = \frac{(-1)^j q^{j(j - 1)/2 + a_i(1 - t) + t j} (q^{1 - \mathbf{b} + a_i - j}; q)_{\mathbf{m} + k - n_i} (q^{- a_i + \boldsymbol{\alpha} + j - 1}; q)_1 (q^{1 + a_i + \boldsymbol{\beta}-n_i + k - j}; q)_1}{(q^{a_i - \mathbf{a}_{\mathbf{[i]}} - j}; q)_{\mathbf{n}_{[i]} + k - n_i + 1} (q; q)_j (q; q)_{k - j}},
\end{align*}
or
\begin{equation} \label{EQ:defgamma5}
\gamma_{i, j}^{k + n_i} = \frac{(-1)^j q^{j (j - 1)/2 + a_i (1 - t) + t j} (q^{1 - \mathbf{b} + a_i - j}; q)_{\mathbf{m} + k} (q^{- a_i + \boldsymbol{\alpha} + j - 1}; q)_1 (q^{1 + a_i + \boldsymbol{\beta} + k- j}; q)_1}{(q^{a_i - \mathbf{a}_{\mathbf{[i]}} - j}; q)_{\mathbf{n}_{[i]} + k + 1} (q; q)_j (q; q)_{k + n_i - j}}.
\end{equation}
Furthermore, we put $\gamma_{i, j}^{k + n_i} = 0$ if $k + n_i < 0$ as before. 
Using this convention, we again  have 
\[
S(z) = \sum_{k = - n_{\max}}^\infty z^k \sum_{i = 1}^r \sum_{j = 0}^{k + n_i} \gamma_{i,j}^{k + n_i}.
\]
Note that  $k \geq - m_{\min}$ for each $i = 1, \ldots, r$ if $- n_{\max} \geq - m_{\min}$.  Otherwise, if $- n_{\max} < - m_{\min}$, 
then we may write decomposition similar to \eqref{EQ:targetsum3} but with $ \gamma_{i, j}^{k + n_i}$ defined in \eqref{EQ:defgamma5}
\[
S(z) = \sum_{k = - n_{\max}}^{- m_{\min} - 1} \mu_k z^k + S_1(z),
\]
where 
\[
S_1(z) = \sum_{k = - m_{\min}}^\infty \mu_k z^k, \quad \quad \mu_k = \sum_{i = 1}^r \sum_{j = 0}^{k + n_i} \gamma_{i, j}^{k + n_i}. 
\]
Next, for each $k \in \Z$ define the functions by the same expression as \eqref{eq:fF-defined}:
\begin{equation} \label{EQ:residuefunction}
\tilde{f}_k(z) = - \frac{( z q^{1 - \mathbf{b}}; q)_{k + \mathbf{m}} (z q^{1 + \boldsymbol{\beta} + k}; q)_1 (z^{-1} q^{\boldsymbol{\alpha} - 1}; q)_1}{(z q^{- \mathbf{a}}; q)_{k + \mathbf{n}+1}} \quad \text{ and } \quad F_k(z) = z^{-t} \tilde{f}_k(z).
\end{equation}
The only difference with \eqref{eq:fF-defined}  is  unequal sizes of the vectors $\mathbf{a}$ and $\mathbf{b}$.   Computing the residues at non-zero poles 
\[
z = q^{a_i - j}, \quad \ i =1, \ldots, r, \quad \text{ and } \quad \ j = 0, \ldots, k + n_i. 
\]
present whenever $k + n_i + 1 > 0$,  we have in comparing with \eqref{EQ:defgamma5} that
\[
\res_{z = q^{a_i-j}}F_k(z) =  \gamma_{i,j}^{k+n_i}. 
\]

As $z\to\infty$, we obtain
\[
F_k(z) \sim \sum_{j = - \infty}^{M + v - N - r - t- (r - s)k} C_j(k) z^j,
\]
which implies that the coefficient $C_{-1}(k) = 0$ if $M + v - N - r - t- (r - s)k < -1$, {\it i.e.}, when $k>p'$, where $p'=\lfloor(M+v-N-r-t+1)/(r-s)\rfloor$. 

The residue at $0$ is unaffected by the sizes of the vectors $\mathbf{a}$ and $\mathbf{b}$ and from  item (ii) of Corollary~\ref{COR:qBeukersJouhet2}, we conclude as before that  the residue of $F_k(z)$ at $z = 0$ is $0$ if $u+t\le0$ and  $\tilde{Q}_{u+t-1}(q^k)$ if $u + t \geq 1$, where $\tilde{Q}_{u+t-1}(y) = \sum_{i = 0}^{u + t - 1} b_{u + t - 1, i} y^i$ is a polynomial of degree $u + t - 1$ in $y$. Using the previous convention  $\tilde{Q}_{-\ell}(y) \equiv 0$ for each $\ell = 1, 2, \ldots$, we obtain 
\[
\mu_k = \sum_{i = 1}^r \sum_{j = 0}^{k + n_i} \gamma_{i, j}^{k + n_i} = C_{-1}(k) - \tilde{Q}_{u + t - 1}(q^k).
\]
Thus, when $u + t \leq 0$, we have
\[
S_1(z) = \sum_{k = - m_{\min}}^{\infty}\mu_k z^k =\sum_{k = - m_{\min}}^{p'} C_{-1}(k) z^k.
\]
The above sum is empty and equals  $0$ if $p' < - m_{\min}$.  On the other hand, if $u + t \geq 1$ and $p' \geq - m_{\min}$, then 
\[
S_1(z) = \sum_{k = - m_{\min}}^\infty \mu_k z^k  = \sum_{k = - m_{\min}}^{p'} C_{-1}(k) z^k - z^{- m_{\min}}\sum_{j = 0}^{u + t - 1} \frac{b_{u + t - 1, j}}{1 - z q^j},
\] 
while if $u + t \geq 1$ and $p' < - m_{\min}$, then
\[
S_1(z) = \sum_{k = - m_{\min}}^\infty \mu_{k} z^k = - z^{- m_{\min}}\sum_{j = 0}^{u + t - 1} \frac{b_{u+t-1,j}}{1 - z q^j}. 
\]
Summarizing all those cases, we get 
\[
S(z) = \sum_{k = - n_{\max}}^{- m_{\min} - 1} \mu_k z^k + S_1(z) = \frac{1}{(z; q)_{(u + t)_{+}}} \sum_{j = - n_{\max}}^{\max(- m_{\min} - 1, p') + (u + t)_{+}} \delta_j z^j. 
\]

Next, we deduce an explicit formula for the coefficient $\delta_j$. As \eqref{EQ:defgamma5} only differs from \eqref{EQ:defgamma3} by the size of the vector $\mathbf{b}$,
we obtain a formula for $\mu_k$ similar to \eqref{EQ:muexplicitformula}, but with modified dimension of $\phi$:
\begin{align}\label{EQ:muexplicitformula-conl} 
\mu_k & = \sum_{i = 1}^r \sum_{j = 0}^{k + n_i} \gamma_{i, j}^{k + n_i}  \nonumber \\
& = \sum_{i = 1}^r \frac{q^{a_i(1-t)}  (q^{1 + a_i + \boldsymbol{\beta} + k}; q)_1 (q^{-a_i+\boldsymbol{\alpha} -1}; q)_1 (q^{1 - \mathbf{b} + a_i}; q)_{\mathbf{m} + k}}{(q; q)_{k + n_i} (q^{a_i - \mathbf{a}_{[i]}}; q)_{\mathbf{n}_{[i]} + k + 1}}  \nonumber \\
& \quad \ \times {_{r+s + u + v}\phi_{r+s + u + v - 1}}\left(\left.\!\!\begin{array}{c} q^{- k - n_i}, q^{\mathbf{b} - a_i}, q^{\mathbf{a}_{[i]} - a_i - \mathbf{n}_{[i]} - k}, q^{- a_i - \boldsymbol{\beta} - k}, q^{-a_i + \boldsymbol{\alpha}} \\ q^{\mathbf{b} - a_i - \mathbf{m} - k}, q^{1 - a_i + \mathbf{a}_{[i]}}, q^{-1 - a_i - \boldsymbol{\beta} -  k}, q^{- a_i + \boldsymbol{\alpha} - 1} \end{array} \right| q^{N - M + r-1 + t-v} \!\right),
\end{align}
where each term with $k + n_i < 0$ vanishes and 
$S(z) = \sum_{k=-n_{\max}}^\infty \mu_kz^k$. Using the Gauss expansion
\[
(z; q)_{(u+t)_{+}} = \sum_{j = 0}^{(u+t)_{+}} \left[ \begin{matrix} (u+t)_{+} \\ j \end{matrix} \right]_{q} q^{j (j - 1)/2} (-z)^j,
\]
we multiply  both sides of \eqref{EQ:generalizedqBeukersJouhet2} by $(z; q)_{(u+t)_{+}}$ and then get
$$
\sum_{j = 0}^{ (u + t)_{+}} \left[ \begin{matrix} (u+t)_{+} \\ j \end{matrix} \right]_{q} q^{j (j - 1)/2} (-z)^j \sum_{k = - n_{\max}}^\infty \mu_k z^k = \sum_{k = - n_{\max}}^{K}\delta_k z^k. 
$$
Multiplying both sides by $z^{n_{\max}}$, changing $k+n_{\max}\to k$ and writing $\hat{\mu}_k = \mu_{k - n_{\max}}, \hat{\delta}_k = \delta_{k - n_{\max}}$, we then get 
\[
\sum_{j = 0}^{ (u + t)_{+}} \left[ \begin{matrix} (u+t)_{+} \\ j \end{matrix} \right]_{q} q^{j (j - 1)/2} (-z)^j \sum_{k =0}^\infty \hat{\mu}_k z^k = \sum_{s = 0}^{\infty}z^{s}\sum_{j + k = s} \left[ \begin{matrix} (u+t)_{+} \\ j \end{matrix} \right]_{q} q^{j (j - 1)/2} (-1)^j \hat{\mu}_k \:=\!\! \sum_{s = 0}^{K+ n_{\max}} \hat{\delta}_s z^s
\]
so that 
$$
\hat{\delta}_s=\delta_{s-n_{\max}}=\sum_{j + k = s} \left[ \begin{matrix} (u+t)_{+} \\ j \end{matrix} \right]_{q} q^{j (j - 1)/2} (-1)^j \hat{\mu}_k=\sum_{j=0}^{\min(s,(u+t)_{+})} \left[ \begin{matrix} (u+t)_{+} \\ j \end{matrix} \right]_{q} q^{j (j - 1)/2} (-1)^j \hat{\mu}_{s-j}
$$
for $s=0,\ldots,K+n_{\max}$.  Set $k=s-n_{\max}$. It implies that 
\begin{multline*}
\delta_{k}=\sum_{j=0}^{\min(k+n_{\max},(u+t)_{+})} \left[ \begin{matrix} (u+t)_{+} \\ j \end{matrix} \right]_{q} q^{j (j - 1)/2} (-1)^j \mu_{k-j}
\\
=\!\!\!\!
\sum_{j=\max(-n_{\max},k-(u+t)_{+})}^{k} \left[ \begin{matrix} (u+t)_{+} \\ k-j \end{matrix} \right]_{q} q^{(k-j) (k-j -1)/2} (-1)^{k-j} \mu_{j}.   
\end{multline*}
Finally substituting \eqref{EQ:muexplicitformula-conl} for $\mu_k$, we arrive at \eqref{eq:qdelta_k}.  $\hfill\square$

\medskip

\noindent Below is an explicit  example of the right-hand side of \eqref{EQ:generalizedqBeukersJouhet2} 
with  $s = 1$ and $r = 3$. 
\begin{example} \label{EX: qconfluentex1}
Set $\mathbf{m} = (2), \mathbf{n} = (1,2,2),  \mathbf{a} = (1, 1/2, 1/3), \mathbf{b} = (2), \boldsymbol{\alpha} = (1/5)$, and $\boldsymbol{\beta} = (1/7)$. Then the right-hand side of~\eqref{EQ:generalizedqBeukersJouhet2} takes the form 
\[
\frac{\left(-q^{93/35}-1\right) }{q^{58/35} z}+\frac{\left(1-q^{-17/15}\right)
   \left(1-q^{-11/21}\right)}{(1-q^{-1/6}) z^2}+\frac{\left(1-q^{-13/10}\right) \left(1-q^{-5/14}\right)}{(1-q^{1/6}) z^2}.
\] 
\end{example}


\begin{thebibliography}{99}

\bibitem{Bailey64} W.N.\:Bailey, Generalized Hypergeometric Series, Cambridge Tracts in Mathematics and Mathematical Physics, 2nd edn. Stechert-Hafner Inc., New York (1964). \url{https://archive.org/details/bwb_P8-ABP-096}

\bibitem{Bailey} W.N.\:Bailey, On certain relations between hypergeometric series of higher order, J. London Math. Soc. Volume 8 (1933), 100--107.
\url{https://doi.org/10.1112/jlms/s1-8.2.100}

\bibitem{BJ} F.\:Beukers and F.\:Jouhet, Duality relations for hypergeometric series, Bull. London Math. Soc. 47(2015), 343--358.
\url{https://doi.org/10.1112/blms/bdv009}

\bibitem{Darling} H.B.C.\:Darling, On certain relations between hypergeometric series of higher orders, Proc. London Math. Soc. Volume~34(1932), 323--339.
\url{https://doi.org/10.1112/plms/s2-34.1.323}

\bibitem{Ebisu}A.\:Ebisu, Three-term relations for the hypergeometric series, Funkc. Ekvac., Volume 55(2012), 255--283.
\url{https://doi.org/10.1619/fesi.55.255}

\bibitem{CKP2021}A.\:Çetinkaya, D.B.\:Karp and E.G.\:Prilepkina, Hypergeometric functions at unit argument: simple derivation of old and new identities, SIGMA 17 (2021), 098.
\url{https://doi.org/10.3842/SIGMA.2021.098}

\bibitem{FKY2}R.\:Feng, A.\:Kuznetsov, F.\:Yang, A short proof of duality relations for hypergeometric functions, J. Math. Anal. Appl. Volume 443(2016), 116--122.
\url{https://doi.org/10.1016/j.jmaa.2016.05.020}


\bibitem{GR04} G.\:Gasper and M.\:Rahman, Basic Hypergeometric Series, Encyclopedia of Mathematics and Its Applications, Vol. 96, Cambridge University Press, 2004.
\url{https://doi.org/10.1017/CBO9780511526251}

\bibitem{Gorelov2010} V.A.\:Gorelov, On algebraic identities between generalized hypergeometric functions, Mathematical Notes, Volume 88:4(2010), 487--491.
\url{https://doi.org/10.4213/mzm8522}

\bibitem{KC02} V.\:Kac and P.\:Cheung, Quantum Calculus, Universitext. Springer, New York (2002).
\url{https://doi.org/10.1007/978-1-4613-0071-7}

\bibitem{KKK}S.I.\:Kalmykov, D.\:Karp and A.\:Kuznetsov, A new identity for the sum of products of generalized basic hypergeometric functions, Ramanujan J. 61(2023), 391--414.
\url{https://doi.org/10.1007/s11139-022-00598-w}

\bibitem{KarpAAM2026} D.\:Karp, Polynomial perturbations of Euler’s and Clausen’s identities,
Adv. Appl. Math., Volume 175, April 2026, 103042.
\url{https://doi.org/10.1016/j.aam.2026.103042}

\bibitem{KK}D.\:Karp and A.\:Kuznetsov, A new identity for a sum of products of the generalized hypergeometric functions, Proc. Amer. Math. Soc., Volume 149, Number 7, 2021, 2861--2870.
\url{https://doi.org/10.1090/proc/14803}

\bibitem{KPSIGMA2016}D.B.\:Karp and E.G.\:Prilepkina, Hypergeometric differential equation and new identities for the coefficients of N{\o}rlund and B\"{u}hring, SIGMA 12 (2016), 052, 23 pages.
\url{https://doi.org/10.3842/SIGMA.2016.052}

\bibitem{KPJMS2018}D.B.\:Karp and E.G.\:Prilepkina, An inverse factorial series for a general gamma ratio and related properties of the N{\o}rlund-Bernoulli polynomials, J. Math. Sci., Vol. 234, No. 5 (2018), 680--696.
\url{https://doi.org/10.1007/s10958-018-4036-1}

\bibitem{Nemes}G.\:Nemes, Generalization of Binet's Gamma function formulas, Integral Transforms Spec. Funct., Volume 24, Issue 8(2013), 597--606.
\url{https://doi.org/10.1080/10652469.2012.725168}

\bibitem{Nesterenko1995} Yu.V.\:Nesterenko, Hermite--Pad\'{e} approximants of generalized hypergeometric functions, Russian Acad. Sci. Sb. Math., Volume 83:1 (1995), 189--219.
\url{https://doi.org/10.1070/SM1995v083n01ABEH003587}

\bibitem{NIST}F.W.J.\:Olver, D.W.\:Lozier, R.F.\:Boisvert and C.W.\:Clark (Eds.) NIST Handbook of Mathematical Functions, Cambridge University Press, 2010.
\url{https://dlmf.nist.gov}

\bibitem{Slater} L.J.\:Slater, Generalized Hypergeometric Functions, Cambridge University Press, Cambridge (1966).
\url{https://archive.org/details/generalizedhyper0000unse}

\bibitem{Yamaguchi} Y.\:Yamaguchi, Three-term relations for basic hypergeometric series, J. Math. Anal. Appl. 464 (2018), no. 1,
662--678.
\url{https://doi.org/10.1016/j.jmaa.2018.04.021}

\end{thebibliography}
\end{document}